\documentclass[12pt,letterpaper]{amsart}
\usepackage{epsfig}
\usepackage{amsmath}
\usepackage{amssymb}
\usepackage{amsthm}
\usepackage{indentfirst}
\usepackage{xspace}
\usepackage{multirow}
\usepackage[toc,page]{appendix}
\usepackage{hyperref}
\usepackage{xcolor}
\definecolor{darkred}{rgb}{0.5,0.15,0.15}
\hypersetup{colorlinks=true,urlcolor=darkred,linkcolor=darkred,citecolor=darkred}
\usepackage{verbatim}
\usepackage[letterpaper,margin=1in,headheight=15pt]{geometry}
\usepackage{mathpazo}
\usepackage{tikz-cd}
\usepackage{booktabs}
\usepackage{framed}
\usepackage{float} %Specify position of figures
\definecolor{shadecolor}{rgb}{0.85,0.85,0.85}
\usepackage{wrapfig}

\usepackage{amssymb,amsmath,amsthm}
\usepackage{verbatim}
\usepackage{centernot}
\usepackage{stmaryrd}
\usepackage{enumerate}

\newtheorem{thm}{Theorem}[section]
\newtheorem{theorem}{Theorem}
\newtheorem*{thm*}{Theorem}

\newtheorem{lemma}{Lemma}
\newtheorem{prop}[thm]{Proposition}

\newtheorem{cor}{Corollary}[section]
\newtheorem{defn}{Definition}
\theoremstyle{definition}
\newtheorem{ex}{Example}[section]
\theoremstyle{remark}
\newtheorem*{rmk}{Remark}

\newcommand{\C}{\mathbb C}
\newcommand{\R}{\mathbb R}

\newcommand{\dd}[1]{\frac{\partial}{\partial #1}}

\makeatletter
\newcommand{\xmapsto}[2][]{\ext@arrow 0599{\mapstofill@}{#1}{#2}}
\def\mapstofill@{\arrowfill@{\mapstochar\relbar}\relbar\rightarrow}
\makeatother

\theoremstyle{remark}

\theoremstyle{definition}

\newcommand{\fixme}[1]{{\color{blue}{\tt [#1]}}}

\begin{document}

\title{Symplectic neighborhood of crossing symplectic submanifolds}

\author{Roberta Guadagni}

\setcounter{page}{1}

\begin{abstract}
This paper presents a proof of the existence of standard symplectic coordinates near a set of smooth, orthogonally intersecting symplectic submanifolds. It is a generalization of the standard symplectic neighborhood theorem. Moreover, in the presence of a compact Lie group $G$ acting symplectically, the coordinates can be chosen to be $G$-equivariant.
\end{abstract}

\maketitle

\section*{Introduction}

The main result in this paper is a generalization of the symplectic tubular neighborhood theorem (and the existence of Darboux coordinates) to a set of symplectic submanifolds that intersect each other orthogonally.
This can help us understand singularities in symplectic submanifolds. Orthogonally intersecting symplectic submanifolds (or, more generally, positively intersecting symplectic submanifolds as described in the appendix) are the symplectic analogue of normal crossing divisors in algebraic geometry. Orthogonal intersecting submanifolds, as explained in this paper, have a standard symplectic neighborhood. Positively intersecting submanifolds, as explained in the appendix and in \cite{mtz_nc}, can be deformed to obtain the same type of standard symplectic neighborhood. 

The result has at least two applications to current research: it yields some intuition for the construction of generalized symplectic sums (see \cite{mtz}), and it describes the symplectic geometry of degenerating families of K\"ahler manifolds as needed for mirror symmetry (see \cite{toricdeg}). The application to toric degenerations is described in detail in the follow-up paper \cite{mymain}. 

While the proofs are somewhat technical, the result is a natural generalization of Weinstein's neighborhood theorem. Given a symplectic submanifold $X$ of $(M,\omega)$, there exists a tubular neighborhood embedding $\phi:NX\rightarrow M$ defined on a neighborhood of $X$. Moreover, an appropriate choice of connection $\alpha$ on $NX$ determines a closed 2-form $\omega_\alpha$, non-degenerate in a neighborhood of $X$ (see Section \ref{sectiononbundles}). Thanks to Weinstein's symplectic neighborhood theorem \cite{weinstein}, the tubular neighborhood embedding can be chosen to be symplectic.
The aim of the present work is to prove an analogous statement for a union of symplectic submanifolds $\{X_i\}_{i\in\mathcal I}$ of $(M,\omega)$. We require the intersection to be orthogonal, that is, $N(X_i\cap X_j)=NX_i\oplus^{\perp}NX_j$, for all $i,j$, at all intersection points (see Definition \ref{orthog}). 

\begin{rmk}
The generalization from one to multiple submanifolds is not straightforward, because there is no smooth retraction of a neighborhood of $\bigcup_{i\in\mathcal I} X_i$ onto $\bigcup_{i\in\mathcal I} X_i$ when $|\mathcal I|>1$. The goal of this paper is to overcome such difficulty.
\end{rmk}

The result can be summarized as follows (see Corollary \ref{maincor}, and see Theorem  \ref{equivsymplum} for the equivariance):

\begin{theorem}\label{thm1}
Let $\{X_i\}_{i\in\mathcal I}$ be a finite family of orthogonal symplectic submanifolds in $(M,\omega)$. Then there exist connections $\alpha_i$ on $NX_i$ and symplectic neighborhood embeddings $\phi_i:(NX_i,\omega_{\alpha_i})\rightarrow (M,\omega)$ which are pairwise compatible. 

In the presence of a group $G$, the embeddings can be chosen to be equivariant with respect to the linearized action.
\end{theorem}

If we assume that the manifolds are of real codimension $2$ (e.g. algebraic divisors), compatibility can be expressed as follows: for all $i,j$, for $X_{ij}=X_i\cap X_j$, there is a symplectic neighborhood embedding $\phi_{ij}:(NX_{ij},\omega_{\alpha_i+\alpha_j})\rightarrow (M,\omega)$, which factors as $\phi_{ij}=\phi_i\circ \xi_{ij}^i=\phi_j\circ\xi_{ij}^j$, where $\xi_{ij}^i:NX_{ij}\cong (NX_i\oplus NX_j)|_{X_{ij}} \rightarrow NX_i$ is a morphism of symplectic vector bundles covering some symplectic neighborhood map $\varphi_{ij}^i: N_{X_i}X_{ij}\rightarrow X_i$
\footnote{Notice that $NX_{ij}=(NX_i\oplus NX_j)|_{X_{ij}}$ can be naturally thought of as a bundle over $NX_j|_{X_{ij}}$; also notice that $NX_j|_{X_{ij}}=N_{X_i}X_{ij}$ (by transversality) so that a neighborhood map for $X_{ij}$ inside $X_i$ is a map $\varphi_{ij}^i: NX_j|_{X_{ij}}\rightarrow X_i$; therefore, it makes sense to ask that $\xi_{ij}^i$ is a bundle map covering $\varphi_{ij}^i$.} (and similarly for $j$) while $\phi_i, \phi_j$ are symplectic tubular neighborhood maps.

\begin{figure}
\vspace{-.0in}
\includegraphics[width=3in]{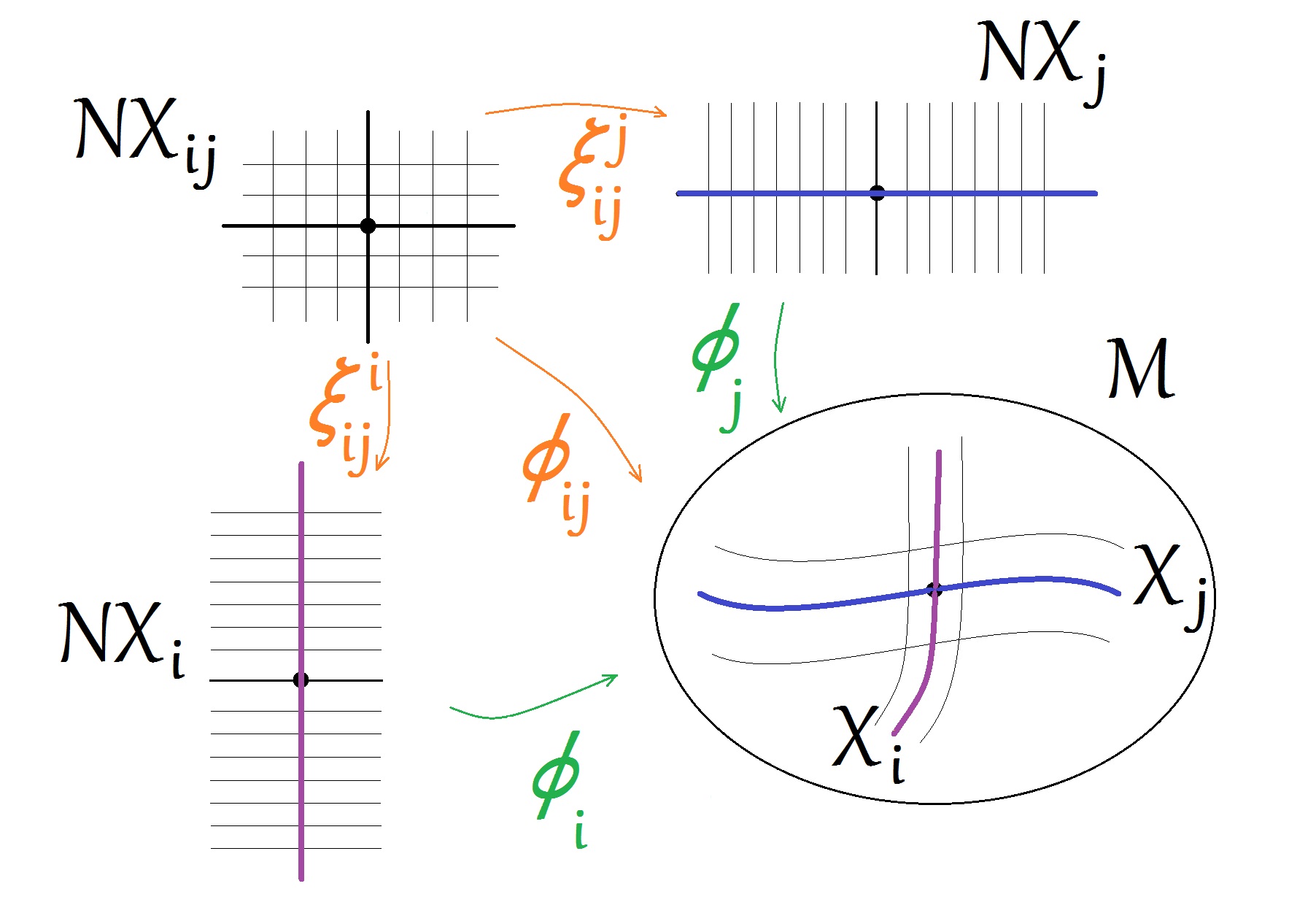}
\vspace{-.0in}
\caption{The compatibility of $\phi_i,\phi_j$ is given by the existence of $\phi_{ij},\xi_{ij}^i,\xi_{ij}^j$.
}\label{syzpic} \vspace{-.1in}
\end{figure}

The result (and the more general compatibility conditions) can be expressed in terms of the auxiliary construction of a symplectic plumbing. A plumbing is an appropriate gluing of the normal bundles $NX_i$ for all $i$, and if we keep track of some extra data we can make sure it inherits a symplectic form. The content of Theorem \ref{main} and Theorem \ref{equivsymplum} can be summarized as follows:
\begin{theorem}[reformulation of Theorem \ref{thm1}]
Given a family of orthogonally intersecting submanifolds $\{X_i\}_{i\in\mathcal I}$ in $(M,\omega)$, there exists a symplectic space given by a union of all the normal bundles, called plumbing. The union $\bigcup_i X_i$ admits a neighborhood in $M$ which is symplectomorphic to a neighborhood of the zero section in the plumbing.

In the presence of a $G$-action, the plumbing inherits a linearized $G$-action and the symplectomorphism can be chosen to be equivariant.
\end{theorem}
This implies Theorem \ref{thm1}. 
\vspace{.1in}

\textbf{Outline of proof.}
For symplectic submanifolds $\{X_i\}_{i\in \mathcal I}$, $X_i\subset (M,\omega)$, the symplectic form $\omega$ on a neighborhood of $\bigcup_i X_i$ only depends on the restriction of $\omega$ to ${TM|_{ \bigcup_i X_i}}$. This is the content of Lemma \ref{lemma1}. It is a generalization of the case of one divisor, and it is done by induction. 

The rest of the paper gives an explicit construction of a local model for such a symplectic form. This is done in Theorem \ref{main}.
The strategy is to first construct an appropriate space containing $\bigcup_i X_i$, which is diffeomorphic to a neighborhood of $\bigcup_i X_i$ in $M$. This can be done by constructing a plumbing, that is, an appropriate glueing of the normal bundles $NX_i$ for all $i$. 
Unfortunately, the construction doesn't run as smoothly when introducing symplectic structures. Therefore, we give a definition of plumbing which is more rigid than usual and requires some extra data (see Definition \ref{plumdata}). This set of rigid data is needed in order to use the result in symplectic settings. The fact that such data always exist is proved in Theorem \ref{plumbingsexist}.
We then show that if a plumbing is built symplectically (as in Definition \ref{symplumdata}), then it admits a symplectic form, which only depends on the symplectic bundle structure of each $NX_i$, together with a choice of compatible connections for each $NX_i$. This is the content of Lemma \ref{sumofbundles}. The fact that one can always find connections which have the right compatibility is proved in Lemma \ref{nconnections}.

The main theorem (Theorem \ref{main}) states that a symplectic plumbing can be constructed for any choice of orthogonal symplectic submanifolds, and it can always be symplectically embedded as a neighborhood of such submanifolds in $M$. 
The proof follows the outline of the proof of the symplectic neighborhood theorem: first, construct a smooth embedding of the plumbing (see Proposition \ref{plumbingsexist}), whose derivative is the identity along each $X_i$; then, apply Lemma \ref{lemma1} to make it a symplectomorphism. The case of multiple submanifolds requires more care than the classical case, due to the compatibility requirements (the main problem is: there isn't a smooth retraction onto $\bigcup_i X_i$; so we need to use separate retractions onto $X_i$ for each $i$, and insure compatibility at every step).

The last section of the paper extends the results in the presence of a compact Lie group acting symplectically. 
When such $G$ acts on $(M, \omega)$ and preserves each submanifold $X_i$, we construct a $G$-action on the plumbing (given by linearizing the action on each $NX_i$). This is done in Lemma \ref{equivplum}. Finally, in Theorem \ref{equivsymplum}, we show that, if $G$ acts symplectically on $M$, the right choice of connections makes the plumbing a symplectic $G$-space, and the symplectic embedding of the plumbing into $M$ can be chosen to be equivariant.

\textbf{Acknowledgments.} Many thanks to Mark McLean and Mohammad Tehrani for helpful conversation, and to Tim Perutz for his advice and  patience. This work was partly supported by the NSF through grants DMS-1406418 and CAREER-1455265.

\section{Symplectic forms on neighborhoods of crossing symplectic submanifolds}

Let $\omega_1, \omega_2$ be symplectic forms on $M$.
Let $\{X_i\}_{i\in I}$ be a collection of symplectic submanifolds of both $(M,\omega_1)$ and $(M,\omega_2)$. Throughout the paper we'll assume that all intersections are transverse.
What follows is a version of Moser's argument for symplectic forms agreeing on transversal symplectic submanifolds.

\begin{defn}\label{agreealong}
Two forms $\alpha,\beta$ on $M$ are said to \textbf{agree along} a subset $S\subseteq M$ whenever $\alpha_x=\beta_x$ on $T_xM$ for all $x\in S$.
\end{defn}

\begin{lemma}\label{lemma1}
Assume that $\omega_1$, $\omega_2$ are symplectic forms on $M$ that agree along $\bigcup_i X_i$. Then there exist open neighborhoods $\mathcal{U}, \mathcal{V}$ of $\bigcup_i X_i$ and a diffeomorphism $\phi:\mathcal{U}\rightarrow \mathcal{V}$ such that $$\phi|_{\bigcup_i X_i}=id \;\ \ \ \ \ \ \  \phi^*\omega_2=\omega_1.$$
\end{lemma}

\begin{proof}

We will approach the proof by induction: for all $0\leq k\leq n$ we'll find: 

\begin{itemize}
\item a 2-form $\omega^{(k)}$ such that $\omega^{(k)}=\omega_1=\omega_2$ along all $X_i$'s, $\omega^{(k)}=\omega_1$ in a neighborhood of $\bigcup_{i=1}^k X_i$;
\item $f_k$ such that $f_k^*(\omega^{(k-1)})=\omega^{(k)}$, $f_k=id$ on $X_i$ for all $i$, and $f_k=id$ on a neighborhood of $\bigcup_{i=1}^{k-1}X_i$.
\end{itemize}

The starting point will be $\omega^{(0)}=\omega_2$, and at the end $\phi=f_1\circ\ldots\circ f_n$ will satisfy $\phi^*\omega_2=\omega^{(n)}=\omega_1$.

Assume, for $k>0$, that $\omega^{(k-1)}$ has been built.
To construct $f_k$, we'll find a 1-form $\sigma_k$ such that
$d\sigma_k=\omega^{(k-1)}-\omega_1$ on a neighborhood of $X_k$, $\sigma_k=0$ along all $X_i$'s and $\sigma_k=0$ on a neighborhood of $\bigcup_{i=1}^{k-1} X_k$;
Moser's theorem applied to $\sigma_k$ yields $f_k:\mathcal U_k\rightarrow \mathcal V_k$ and $\omega^{(k)}=f_k^*(\omega^{(k-1)})$ as desired.
\vspace{.3cm}

To construct $\sigma_k $, consider the exponential map with respect to a metric $g$, and let $exp_k:N_MX_k\cong T_{X_k}M^{\perp_g}\rightarrow M$ be the restriction of this map to the normal bundle of $X_k$ (Really, we are only considering this map for vectors of length at most $\delta$ so that $exp_k$ is a diffeomorphism. The slight abuse of notation reflects an effort to keep the notation from getting too heavy). Let $\mathcal N_k$ be the image of such map, i.e. a tubular neighborhood of $X_k$. We make sure to pick $g$ such that $exp_k$ restricts to a map $exp_k|_{N_{X_j}(X_k\cap X_j)}: {N_{X_j}(X_k\cap X_j)}\rightarrow X_j$ for all $j\neq k$ (we'll construct such a metric later in this paper; see Corollary \ref{totallygeodesicmetric}).
Now define $\phi_t:\mathcal N_k\rightarrow \mathcal N_k$, the map sending $exp_k(q,v)\mapsto exp_k(q,tv)$. In particular, for all $j\neq k$, $\phi_t$ restricts to a map from $X_j\cap \mathcal N_k$ to itself.

(Notation reminder: when $(q,v)\in T^{vert}_q(M)$,  $m\in M$,  $exp(q,v)\in M$, then $\omega(m;v,w)$ denotes $\omega$ at the point $m$ evaluated on $v,w\in T_m(M)$; similarly for a one form $\sigma(m;v)$ denotes $\sigma$ at the point $m$ evaluated on $v$.)

Now we want to consider the family of one-forms on $M$:
$$\sigma_t(m;v)=(\omega^{(k-1)}-\omega_1)(\phi_t(m); \frac{d}{dt} \phi_t(m), d\phi_t(m)v)$$
and integrate it to get $$\tilde\sigma_k(m)=\int_0^1\sigma_t(m)$$
If $m\in X_j$, then $\phi_t(m)\in X_j$ (due to the choice of metric). Because $(\omega^{(k-1)}-\omega_1)$ vanishes on $X_j$, $\sigma_t(m;v)=0$ for all $v$. Therefore $\tilde\sigma_k(m)=\int_0^1\sigma_t(m)=0$. Now we can extend $\tilde\sigma_k$ to a full neighborhood of $\bigcup_i X_i$ by letting $B_{\epsilon}(X_k)\subset\mathcal N$ be a smaller neighborhood of $X_k$, and define
\[
 \sigma_k =
  \begin{cases} 
      \hfill \tilde\sigma_k    \hfill & \text{ if } x\in B_{\epsilon}(X_k) \\
      \hfill 0 \hfill & \text{ if } x\notin \mathcal N\\
      \hfill \text{} \hfill & \text{smoothly interpolates on } \mathcal N\backslash B_{\epsilon}\\
  \end{cases}
\]
Such $\sigma_k$ is as required and yields a diffeomorphism $f_k:\mathcal U_k\rightarrow \mathcal V_k$. Also notice that $f_k=id$ whenever $\sigma_k=0$, in particular on $\bigcup_i X_i$ and on a neighborhood of $\bigcup_{i=1}^{k-1} X_i$. Let then $\omega^{(k)}=f_k^*\omega^{(k-1)}$: by construction, $\omega^{(k)}=\omega_1=\omega_2$ along all $X_i$'s, $\omega^{(k)}=\omega_1$ in a neighborhood of $\bigcup_{i=1}^k X_i$.

\end{proof}

\begin{rmk}
1. The transversality condition can be somewhat relaxed and the proof of the lemma still works in some cases, think e.g. two lower dimensional submanifolds that are not tangent at the intersection points. Weaker formulations can be found if needed.

2. This lemma doesn't need orthogonality for the intersection of the $X_i$'s.

3. The lemma can be easily modified to include self-crossing submanifolds.

4. It is tempting to avoid induction and look for a proof by first constructing a retraction of a neighborhood of $\bigcup_i X_i$ onto $\bigcup_i X_i$, then using Moser's argument only once on the resulting vector field; however, this argument wouldn't work because such a retraction is not smooth; hence the need for induction.
\end{rmk}

In order to use Lemma \ref{lemma1} to prove the existence of a nice neighborhood of intersecting submanifolds $\{X_i\}_{i\in\mathcal I}$, we need to find a multifold analogue of the following classical theorem, which is used in the classic result of Weinstein \cite{weinstein} (for a detailed discussion of Weinstein's result cfr. \cite{mcdsal}):
\begin{thm*}[Tubular neighborhood theorem]\label{snt}
There exists a map from $NX$ to $M$ such that $f|_X=id$ and $df|_X=id$.
\end{thm*}
Classically, such a map can be obtained as the exponential map of any metric, applied to $T^{\perp_\omega}_XM$ (which is naturally isomorphic to $NX$). This can be much harder to prove (or even state) for $X$ not smooth.

\section{Normal bundles and connections}\label{sectiononbundles}

This section is a review of some basic facts for complex and symplectic vector bundles.

\textbf{Remarks on notation.} For submanifolds $Z\subset X\subset M$, $N_MX$ will indicate the normal bundle to $X$ inside $M$ (sometimes only $NX$ when there is no ambiguity), so that $N_XZ$ is the normal bundle to $Z$ inside $X$, while $NX|_Z$ indicates the restriction to $Z$ of the normal bundle to $X$ inside $M$.

The normal bundle to a symplectic submanifold is naturally a symplectic bundle.
Recall that one can always choose a compatible complex structure on a symplectic bundle: the choice is unique up to isotopy. Moreover, two bundles are isomorphic as symplectic bundles if and only if they are isomorphic as complex bundles (see e.g. \cite{mcdsal}). Therefore we will chose complex structures and consider the normal bundles as complex bundles.

There will be some abuse of notation with respect to connections on normal bundles: given $Y\subset X$ a symplectic submanifold of real codimension $2$, one can obtain a model for the complex normal bundle by letting $NY=U\times_{U(1)}\C$ where $U$ is a principal $U(1)$-bundle and $\alpha$ is a $U(1)$-connection on $U$. Whenever the paper mentions a connection $\alpha$ on $NX$, it can be interchangeably a connection 1-form on a $U(1)$-bundle (for which $NX$ is the associated bundle) or the corresponding Ehresmann connection on $NX$. 
\begin{rmk}
Choosing a $U(1)$-connection $\alpha$ gives the same data as a unitary Ehresmann connection on $NX$, whence the abuse of notation.
\end{rmk}

One can then build a two-form on $U\times \C$:
$$\omega_\alpha=\omega_Y+d(\rho \alpha)+\omega_{\C}$$
where $\rho$ indicates the hamiltonian function generating the standard circle action on $\C$.  Such two-form descends to $U\times_{U(1)}\C$. This is a general fact:

\begin{lemma}
Let $(F,\omega_F)$ be a symplectic manifold, $G\circlearrowright F$ a Hamiltonian action with moment map $\mu$, $P$ a principal $G$-bundle over a symplectic base $(Y, \omega_Y)$, $X=P\times_G F$ the associated bundle, $\alpha$ a principal connection on $P$. Then the 2-form $\omega_F+d\langle\alpha,\mu\rangle$ on $P\times F$ descends to a $2-form$ on $X$. Therefore so does $\Omega=\omega_Y+d\langle\alpha,\mu\rangle+\omega_F$ (here $\omega_Y$ is short for $\pi^*(\omega_Y)$).
\end{lemma}

\begin{proof}
Let $v_\xi$ be the vector field on $P\times F$ corresponding to $\xi\in\mathfrak{g}$. Recall that a form $\eta$ on $P\times F$ is the pullback of a form on $X$ if and only if $\iota_{v_\xi}\eta=\mathcal L_{v_\xi}\eta=0$, which by Cartan's formula  is equivalent to $\iota_{v_\xi}\eta=\iota_{v_\xi}(d\eta)=0$. Since $\omega_F+d\langle\alpha,\mu\rangle$ is closed we will just check that, for every  $\xi\in\mathfrak{g}$,
$\iota_{v_\xi}(\omega_F+d\langle\alpha,\mu\rangle)=0$.

Since $\mathcal{L}_{v_\xi}\alpha=0$, also $\mathcal{L}_{v_\xi}(d\langle\alpha,\mu\rangle)=0$ therefore
$$ \iota_{v_\xi}\circ d\langle\alpha,\mu\rangle=
\iota_{v_\xi}<d\alpha,\mu>-\iota_{v_\xi}(\alpha\wedge d\mu) \text{ by expanding the derivative}$$
$$= -\iota_{v_\xi}(\alpha\wedge d\mu) \text{ because the other term is 0 when $\alpha$ is a principal connection} $$
$$= - <\xi, d\mu> = -d<\xi, \mu>=\iota_{\xi_P}\omega_F$$
so that $\iota_{v_\xi}(\omega_F+d\langle\alpha,\mu\rangle)=\iota_{\xi_P}\omega_F-\iota_{\xi_P}\omega_F=0$.
\end{proof}

We can also find a more explicit description of $\Omega$ on $X$. First note that the tangent bundle of $P$ splits as $T^{vert}P\oplus T^{hor}P$. Let $\pi:P\rightarrow Y$ be the projection, then the horizontal component $T^{hor}P:=\ker \alpha$ is identified by $D\pi$ with $\pi^*TY$, while the vertical component is $T^{vert}P:=\ker D\pi = \mathfrak g$. Thus (omitting notation for pullbacks of projections) we have $T(P\times F)=\mathfrak{g}\oplus\pi^*TY\oplus TF$, and $\mathfrak{g}$ acts on $T(P\times F)$ by $\xi\mapsto (-\xi,0,\xi_F)$, hence $$TX\cong \pi^* TY\oplus TF.$$
In other words, a tangent vector $z\in TX$ may be uniquely rewritten as $z=u^\natural+v$, where $u^\natural$ is the horizontal lift of $u\in TY$ to $TP$ while $v\in TF$. Moreover, the vector $\tilde z=(0,u^{\natural},v)\in \mathfrak{g}\oplus\pi^*TY\oplus TF\cong T(P\times F)$ projects to $z$. This means that one can compute $\Omega(z_1,z_2)$ by evaluating $\Omega(\tilde z_1, \tilde z_2)$.

Now let's describe $\Omega$. First recall that the curvature is $F_\alpha\in\Omega^2(M;P\times_G\mathfrak{g})$. It induces a 2-form via pullback on $P$: $\pi^*F_\alpha=d\alpha+\frac{1}{2}[\alpha\wedge\alpha].$ Notice that $d\langle \alpha,\mu\rangle=\langle d\alpha,\mu\rangle+\langle\alpha\wedge d\mu\rangle$ on $P\times F$; the key observation is that $\alpha(\tilde z_1)=\alpha(\tilde z_2)=0$. Then we can calculate
$$\Omega(z_1, z_2)=\Omega(\tilde z_1, \tilde z_2)=$$ $$=\omega_Y(u_1,u_2)+\langle d\alpha(\tilde z_1, \tilde z_2),\mu\rangle+\frac{1}{2}(\langle\alpha(\tilde z_1), d\mu(\tilde z_2)\rangle-\langle\alpha(\tilde z_2), d\mu(\tilde z_1)\rangle)+\omega_F(v_1,v_2)$$ 
$$=\omega_Y(u_1,u_2)+\langle d\alpha(\tilde z_1, \tilde z_2),\mu\rangle+\omega_F(v_1,v_2)$$
$$= \omega_Y(u_1,u_2)+\langle(\pi^*F_\alpha -\frac{1}{2}[\alpha\wedge\alpha])(\tilde z_1, \tilde z_2),\mu\rangle+\omega_F(v_1,v_2)$$
$$= \omega_Y(u_1,u_2)+\langle \pi^*F_\alpha (\tilde z_1, \tilde z_2),\mu\rangle -\langle \frac{1}{2}[\alpha\wedge\alpha])(\tilde z_1, \tilde z_2),\mu\rangle +\omega_F(v_1,v_2)$$
$$= \omega_Y(u_1,u_2)+\langle F_\alpha(u_1,u_2),\mu\rangle +\omega_F(v_1,v_2).$$

\begin{ex}
If, for instance, $\alpha$ is the flat connection induced by a trivialization, then the form on $Y\times F$ is given by $\omega_Y+ \omega_F$. In fact this is an if and only if.
In slightly different words: $F_\alpha$ is the obstruction to $\omega_F$ pushing forward to a form on $X$.
\end{ex}

\begin{cor}
For a complex line bundle $U\times_{U(1)}\C$, with a connection $\alpha$, over a symplectic base $(Y,\omega_Y)$ the 2-form $\omega_\alpha=\omega_Y+d(\rho \alpha)+\omega_{\C}$ descends to a closed 2-form on $U\times_{U(1)}\C$. This form can be described as $\omega_Y+\rho F_\alpha +\omega_\C$ with respect to a horizontal/vertical splitting due to $\alpha$.
\end{cor}

\begin{comment} 
\begin{lemma}
On $L^*$, 
$$\omega_\alpha=\omega_Y+d( \alpha \rho)$$
where $\alpha$ is the imaginary 1-form on $U$ as above. On a tubular neighborhood of $Y$, this induces a symplectic form.
\end{lemma}

\begin{proof}
Notice that $\alpha$ determines a splitting of $TL$ as $T^{hor}L+T^{vert}L$.

For $w_1, w_2\in TL$, by definition, $\omega_\alpha$ is defined on $U\times \C$, so we can evaluate it by evaluating $\omega_\alpha(w_1^*, w_2^*)$ where $w_1^*,w_2^*$ is any choice of pullback of $w_1,w_2$. Write $w_i=u_i+v_i$ where $u_i\in T^{hor}L$, $v_i\in T^{vert}L\cong \C$. Then $u_i$ projects to $\tilde u_i\in TL$ which in turn pulls back to $\overline{u_i}\in T^{hor}U\subset TU$ (recall that $\alpha$ also induces a splitting of $TU$ into horizontal and vertical). We can choose $w_i^*=\overline{u_i}+v_i\in TU\times \C$. Then $$\omega_\alpha(w_1^*, w_2^*)=\omega_Y(w_1^*, w_2^*)+d(\rho \alpha)(w_1^*, w_2^*)+\omega_{\C}(w_1^*, w_2^*)$$ $$=\omega_Y(u_1, u_2)+d(\rho \alpha)(w_1^*, w_2^*)+\omega_{\C}(v_1, v_2)
$$
\end{proof}

Note that $\omega_Y+d( \alpha \rho)$ can be expanded to get 
$$\omega_\alpha=\omega_Y+\rho d\alpha+ d\rho \wedge\alpha$$
where the last term is equivalent to projecting $TM$ onto the vertical space through $\alpha$, then applying $\omega_\C$ (the form on the fibre). For simplicity, this will sometimes also be expressed by the notation
$$\omega_\alpha=\omega_Y+\rho d\alpha+ \omega_\C.$$
\end{comment}

Similarly, if $Y\subset M$ has codimension $2k$ and $NY=\bigoplus_{s=1}^k \mathcal L_s$ is a sum of line bundles, then $NY=U\times_{U(1)^k}\C^k$ where $U$ is the sum of unit bundles inside $L_s$'s. If $\alpha_s$'s are $U(1)$-connections, then $\bigoplus_{s=1}^k \alpha_s$ is a connection on $U$ and the induced symplectic form on a tubular neighborhood of $Y$ is 
$$\omega_{\oplus_s\alpha_s}=\omega_Y+\sum_s (\rho_s F_{\alpha_s})+\omega_{\C^k}.$$

\begin{cor}
Let's consider the case of $Y\subset M$ of (real) codimension higher than two. Then $NY$ is a $\C^k$-bundle and the structure group is $G=U(k)$. Given a connection form $\alpha$ on the corresponding principal $G$ bundle $P$, we get an induced symplectic form
$$\omega_\alpha=\omega_Y + \langle F_{\alpha},\mu \rangle +\omega_{\C^k}.$$
\end{cor}

More in general we can define $\omega_{\oplus_s\alpha_s}$ on a sum of bundles, where $\alpha_s$ is the connection on a complex vector bundle of dimension $k_s$ over $Y$: 
$$\omega_{\oplus_s\alpha_s}=\omega_Y+\sum_s \langle F_{\alpha_s},\mu_s\rangle+\omega_{\C^{\sum_s k_s}}.$$

\section{A remark on transversality}
In the rest of the discussion all collections of submanifolds will be transverse in the following sense:
\begin{defn}\label{transverse}
Let $\{X_i\}_{i\in\mathcal I}$ be a finite collection of submanifolds of $M$. Let $X_I=\bigcap_{i\in I}X_i$ for all subsets $I\subseteq \mathcal I$. The collection is called \textbf{transverse} if $NX_I=\oplus_{i\in I}NX_i|_{X_I}$.
\end{defn}

This is stronger than just pairwise transversality. On the other hand, once we move to the symplectic world, we assume orthogonality in the following sense:

\begin{defn}\label{orthog}
The intersection of two symplectic submanifolds $X_1,X_2\subset (M,\omega)$  is \textbf{orthogonal} if there is an orthogonal decomposition $TM=TX_{12}\oplus^{\perp}NX_1\oplus^{\perp}NX_2$.
\end{defn}

\begin{rmk}
This definition in particular implies that the two submanifolds are transverse. If you care about other situations (e.g. two Riemann surfaces intersecting in $\R^6$ with linearly independent tangent spaces) then you should look for slightly different definitions. The results in this paper can be adapted to other reasonable definitions.
\end{rmk}

For a family of orthogonal intersection, we can afford to only make pairwise assumptions, because of the following:

\begin{lemma} 
Let $\{X_i\}_{i\in\mathcal I}$ be a finite collection of symplectic submanifolds of $(M,\omega)$. Assume that, whenever $i\neq j$, $X_i, X_j$ intersect orthogonally. Then $\{X_i\}_{i\in\mathcal I}$ is a transverse collection.
\end{lemma}

\begin{proof}
The key observation is that for orthogonal sums, $NX_i\perp NX_j$ for all $j\neq i$ implies $NX_i\perp \oplus_{j\in J} NX_j$ whenever $i\notin J$. This means $\{NX_i\}_{i\in I}$ is a linearly independent collection of subspaces in $TM$. Pick a point $x\in X_I$. 
Notice that there is an embedding $N_x X_i\overset{\iota_i}\hookrightarrow N_x X_I$ when $i\in I$ (if a vector is orthogonal to $X_i$ then it must also be orthogonal to the subset $X_I$). Now let's consider the map $\oplus_i \iota_i:\oplus_{i\in I} N_x X_i \rightarrow NX_I$. This map is injective since $NX_i\subset NX_I$ are linearly independent subspaces. It is then an isomorphism because of dimensions. We can use the formula for the dimension of an intersection of vector subspaces and get $\dim(T_xX_I)=\dim(\bigcap_{i\in I} T_xX_i)\geq \sum_i \dim TX_i-(k-1)\cdot n=\sum_{i\in I} (n-\dim NX_i)-(k-1)\cdot n=n-\sum_{i\in I} \dim NX_i$ (here $k=\# I$, $n=\dim(TM)$. Then $\dim NX_I =n-\dim(T_xX_I)\leq n-(n- \sum_{i\in I} \dim (NX_i))=\sum_{i\in I} \dim NX_i$.
\end{proof}

\section{Plumbing for smooth manifolds}

This section is devoted to carefully describing smooth plumbings and their embedding. The plumbings that we will consider later are enhanced versions of the smooth ones.

\textbf{Remark on notation:} Throughout the paper, for each $I\subset \mathcal I$, $X_I:=\bigcap_{i\in I} X_i$ will denote the intersection.

\subsection{Plumbing data}
Let $\{X_i\}_{i\in\mathcal I}$ be a finite collection of smooth submanifolds in $M$. Let their intersections be transverse (as in Definition \ref{transverse}).

In order to build and embed a plumbing, we need the following:
\begin{defn}\label{plumdata}
A \textbf{system of tubular neighborhoods} for the family $\{X_i\}_{i\in\mathcal I}$ is a collection of tubular neighborhood maps $$\psi_I^{I'}:N_{X_{I'}}X_I\rightarrow X_{I'}\text{\ \ \ \ \ \ \ }   \forall I'\subset I\subset \mathcal I \text{\ \ \ \ \ \ \ \ } I'\neq \emptyset$$ where the maps are defined on a neighborhood $\mathcal V_I^{I'}$ of $X_I$ in $N_{X_{I'}}X_I$. We require that the restriction of $\psi_I^{I'}$ to $X_{I}$ is the identity. Let $\mathcal U_I^{I'}=Im(\mathcal V_I^{I'})$ be the corresponding neighborhood in $X_{I'}$.
The maps naturally induce smooth retractions $\rho_I^{I'}=\pi\circ (\psi_I^{I'})^{-1}:\mathcal U_I^{I'}\rightarrow X_I$.

We require that $\rho_I^{I'}\circ \rho_{I'}^{I''}= \rho_I^{I''}:U_{I}^{I''}\rightarrow X_I$ for all $I''\subset I'\subset I$.

A \textbf{system of bundle isomorphisms covering the retractions} is a collection of maps $\Phi_I^{I'}:NX_I\rightarrow NX_{I'}|_{\mathcal U_I^{I'}}$ given by a choice of isomorphisms $\widetilde\Phi_I^{I'}:(\rho_I^{I'})^*(NX_{I'}|_{X_I})\rightarrow NX_{I'}|_{\mathcal U_I^{I'}}$. The maps are defined as $$\Phi_I^{I'}=\widetilde\Phi_I^{I'}\circ (\psi_I^{I'},id):NX_I\cong N_{X_{I'}}X_I\oplus NX_{I'}|_{X_I}  \rightarrow NX_{I'}|_{\mathcal U_I^{I'}}.$$

The data of a system of tubular neighborhoods together with a system of bundle isomorphisms is \textbf{compatible} when, for all $I''\subset I'\subset I$:
\begin{itemize}
\item
$\psi_I^{I''}=\psi_{I'}^{I''}\circ (\Phi_I^{I'})|_{X_{I''}}:N_{X_{I''}}X_I\rightarrow N_{X_{I''}}X_{I'}|_{\mathcal U_I^{I'}}\rightarrow X_{I''};$
\item
$\Phi_I^{I''}=\Phi_{I'}^{I''}\circ \Phi_I^{I'}:NX_I\rightarrow NX_{I''} .$

\end{itemize}

We also refer to this set of compatible data as \textbf{plumbing data}. 
\end{defn}

Given the data above, we can define smooth plumbings:
\begin{defn}
Given a finite family of smooth submanifold $\{X_i\}_{i\in\mathcal I}$ where $|\mathcal I|=n$; given a system of tubular neighborhoods and a system of bundle isomorphisms that are compatible, the \textbf{(n)-plumbing} of the family is
$$N(\bigcup_i X_i)=\bigcup_i NX_i / \sim$$
where $\sim$ is the equivalence relation given by $\Phi_I^{i}(x)\sim \Phi_I^{j}(x)$ for all $I$, for all $i,j\in I$. The (n)-plumbing will be referred to as \textbf{plumbing} when there is no ambiguity.
\end{defn}

Let's check that this is a good definition (nothing bad happens at higher intersections):
\begin{lemma}
The inclusion $NX_i\hookrightarrow \bigcup_i NX_i / \sim$ is injective. Therefore, the plumbing is a smooth manifold.
\end{lemma}

\begin{proof}
We need to show that each point of $NX_i$ in a neighborhood of $X_{ij}$ is identified with exactly one point of $NX_j$.
The trouble potentially arises at higher intersection points. Consider a neighborhood of $X_{ijk}$. Consider a point $(z,v_i,v_j,v_k)\in NX_{ijk}\cong NX_i\oplus NX_j \oplus NX_k$. When looking at the equivalence relation from $NX_i$ to $NX_j$ we get 
$$N(X_{i})\ni \Phi_{ij}^i(\Phi_{ijk}^{ij}(z,v_k,v_j, v_i))\sim \Phi_{ij}^j(\Phi_{ijk}^{ij}(z,v_k,v_j, v_i)) \in N(X_{j}).$$
On the other hand if we consider $NX_j$ and $NX_k$ we get
$$N(X_{j})\ni\Phi_{jk}^j(\Phi_{ijk}^{jk}(z,v_k,v_j, v_i))\sim \Phi_{jk}^k(\Phi_{ijk}^{jk}(z,v_k,v_j, v_i))\in N(X_{k})$$
while $NX_i, NX_k$ yield
$$N(X_{i})\ni\Phi_{ik}^i(\Phi_{ijk}^{ik}(z,v_k,v_j, v_i))\sim \Phi_{ik}^k(\Phi_{ijk}^{ik}(z,v_k,v_j, v_i)) \in N(X_{k})$$
so $NX_i$ gets identified with $NX_k$ both directly and through $NX_j$: do the two identifications agree?
We can use the compatibility assumption, $\Phi_{I'}^{I''}\circ \Phi_{I}^{I'}=\Phi_I^{I''}$. All the above relations then reduce to
$$(\Phi_{ijk}^i(z,v_k,v_j,v_i))\sim (\Phi_{ijk}^j(z,v_k,v_i,v_j)) \sim (\Phi_{ijk}^k(z,v_i,v_j,v_k)).$$
\end{proof}

\begin{rmk}[Important]
The data used to construct the (n)-plumbing consists of a set of maps which can be thought of as the data of, for each $X_i$, the (n-1)-plumbing of its intersection with the other submanifolds, plus an embedding of such (n-1)-plumbing into $X_i$.
\end{rmk}

\subsection{Metrics on plumbings}

\begin{lemma}\label{bundlemetric}
Given a smooth submanifold $X\subset M$, one can construct metrics on $NX$ which make $X$ totally geodesic, make the fibres $N_xX$ totally geodesic (for all $x\in X$), and are linear on the fibres $N_xX$ (when there is no ambiguity, we will call this type of metric a \textbf{bundle metric} on $NX$).
\end{lemma}

\begin{proof}
We can freely choose the following data:
\begin{itemize}
\item A metric $g$ on X;
\item A smoothly varying family of linear metrics $g_x$ on the fibres, that is, a bundle metric (it can be easily constructed locally, and local patches can be used to construct a global bundle metric through a partition of unity);
\item A metric connection on $NX$ (this is equivalent to a choice of principal connection on the bundle of orthonormal frames).
\end{itemize}
Now let's define the metric $\lambda$ on $NX$, on a local trivialization of $NX$, as $\lambda=g+g_x$ (with respect to the Ehresmann splitting $T_p(NX)\cong T_xX\oplus N_xX$). We can observe that such local definitions agree on chart intersections and form a metric on the whole bundle.

It is clear then that $\lambda$ makes $X$ and $N_xX$ totally geodesic, because this property can be checked locally. Locally, the bundle looks like $(\mathcal U\times N_xX,g+g_x)$, where $\mathcal U$ is an open set in $X$; with respect to the local trivialization, $g_x$ is independent of $x$.
\end{proof}

The following lemma is the key to constructing well behaved metrics on plumbings.

\begin{lemma}\label{extendplumbmetric}
Given a plumbing, assume that there is a metric $g$ on $\bigcup_i X_i$ (that is, metrics $g_i$ defined on each $X_i$, agreeing on intersections) such that $\psi_I^J=exp_g:N_{X_J}X_I\rightarrow X_J$ for all non-empty $J\subset I\subseteq \mathcal I$.
Moreover let's assume that $(\psi_{ij}^i)^{*}g_i=\lambda_j$ on $N_{X_i}X_{ij}\cong NX_j |_{X_{ij}}$ is a bundle metric as in Lemma \ref{bundlemetric} (in particular it is linear in the fibres of $N_{X_i}X_{ij}$ for all $i,j$). Also, let's assume that, in a neighborhood of a higher intersection $X_{ijk}$, $\lambda_j=(\Phi_{ijk}^{ij})_*(\lambda_j|_{X_{ijk}})$ is the push-forward of the fibrewise metric on $NX_j|_{X_{ijk}}$ by the (local) bundle isomorphism $(\Phi_{ijk}^{ij})|_{N_{X_i}{X_{ijk}}}:N_{X_i}{X_{ijk}}\cong N_{X_{ij}}X_{ijk}\oplus N_{X_{ik}}X_{ijk}\rightarrow N_{X_i}X_{ij}$.

Then there is a metric $\mu$ on $P=N(\bigcup_i X_i)$ such that $exp_\mu=id:NX_i\rightarrow NX_i\subset P$ for all $i$ (and $exp_\mu=id:NX_I\rightarrow NX_I\subset P$). Such $\mu$ is linear in the fibres of $NX_i$ for all $i$.
\end{lemma}

\begin{proof}
For each $i$, let's construct a metric $\mu_i$ on $NX_i$ following Lemma \ref{bundlemetric}. We need choices of:
\begin{itemize}
\item a metric on $X_i$;
\item a fibrewise linear metric;
\item a metric connection.
\end{itemize}
Let's pick the given $g_i$ to be the metric on $X_i$.
To get a fibrewise linear metric, notice that such a metric is already given by $(\psi_{ij}^j)^{*}g_j=\lambda_i$ on $N_{X_j}X_{ij}\cong NX_i|_{X_{ij}}$ at the fibres corresponding to intersections.
Let's extend this to all of $NX_i$ by using the bundle isomorphisms. Consider the bundle isomorphism $\Phi_{ij}^i:NX_{ij}\rightarrow NX_i$. On $NX_{ij}$ we can easily construct a fibrewise metric given by $g$ (by considering on each fibre $N_xX_{ij}\cong N_xX_i\oplus N_xX_j$ the metric $\lambda_i+\lambda_j$). This, in particular, yields a fibrewise linear metric on $NX_{ij}$, considered as a bundle over $N_{X_i}X_{ij}$. Therefore, the push forward yields a fibrewise metric $\tilde\lambda_i^j$ on $NX_i$, in a neighborhood of the intersection with $X_j$.

The fact that, when constructing fibrewise metrics in a neighborhood of $X_j$ or $X_k$ in $X_i$, we get the same result (that is, the fact that $\tilde\lambda_i^j=\tilde\lambda_i^k$) is due to the plumbing property $\Phi_{ijk}^i=\Phi_{ij}^i\circ \Phi_{ijk}^{ij}=\Phi_{ik}^i\circ \Phi_{ijk}^{ik}$ together with the fact that $\tilde \lambda_i^j|_{X_{ij}}=(\Phi_{ijk}^{ij})_*(g), \tilde \lambda_i^k|_{X_{ik}}=(\Phi_{ijk}^{ik})_*(g)$ (therefore $\tilde\lambda_i^j=(\Phi_{ij}^i\circ \Phi_{ijk}^{ij})_*(g)=\Phi_{ik}^i\circ \Phi_{ijk}^{ik}=\tilde\lambda_i^k$).
So now we have a fibrewise metric $\tilde\lambda_i$ on $NX_i$, on a neighborhood of $\bigcup_{j\neq i} X_{ij}$. Let's extend it to a fibrewise metric on all of $X_i$ (there isn't any restriction on the choice of metric away from the intersections: the extension exists by a partition of unity argument).

Now on to finding a connection on $X_i$. In a neighborhood of $X_{ij}$ (for all $j$) there is a natural connection on $NX_i|_{\mathcal U_{ij}^i}\cong NX_{ij}\cong NX_i|_{X_{ij}}\oplus NX_j|_{X_{ij}}$. The latter has a connection $\alpha_{ij}$ given by a direct sum of the metric connections $\alpha_{\lambda_i}$ on $NX_i|_{X_{ij}}\cong N_{X_j}X_{ij}$ and $\alpha_{\lambda_j}$ on $NX_j|_{X_{ij}}\cong N_{X_i}X_{ij}$. This yields a connection $\tilde\alpha_{ij}$ on the bundle $NX_{ij}\rightarrow NX_j|_{X_{ij}}$.
So then the pushforward of $\tilde\alpha_{ij}$ via the isomorphism $\Phi_{ij}^i:NX_{ij}\rightarrow NX_i$ is a connection on $NX_i$, close to a neighborhood of $X_j$. As before, such connection can be constructed close to any intersection, and it's easy to check that $(\Phi_{ij}^i)^*\tilde\alpha_{ij}=(\Phi_{ik}^i)^*\tilde\alpha_{ik}$.
We get a metric connection on $(NX_i,\tilde\lambda_i)$ defined close to a neighborhood of $\bigcup_{j\neq i}X_{ij}$.
Again, by using a partition of unity, let's extend this connection to a metric connection on the whole of $NX_i$.

The data of $\tilde\lambda_i$ on $NX_i$ and a corresponding metric connection yields a metric $\mu_i$ on  $NX_i$. Let's check that $\mu_i=\mu_j$, whenever they are both defined, on $N(\bigcup_i X_i)$. Just notice that $(\Phi_{ij}^i)^(-1)^*\mu_i=(\Phi_{ij}^i)^(-1)^*\mu_j$ by construction of $\mu_i,\mu_j$ for all $i,j$.
This means that all the $\mu_i$ agree and yield a metric $\mu$ on $N(\bigcup_i X_i)$.

The required property $exp_\mu=id:NX_i\rightarrow NX_i\subset P$ for all $i$ (and $exp_\mu=id:NX_I\rightarrow NX_I\subset P$) is then satisfied because clearly from the construction $exp_{\mu_i}=id:NX_i\rightarrow NX_i\subset P$.

\end{proof}

\begin{lemma}\label{plumbmetric}
Given a plumbing, there is a metric $\lambda$ on it such that, for each $i$, $\lambda|_{NX_i}$ has the structure described in Lemma \ref{bundlemetric}.
In particular each $X_i$ is totally geodesic with respect to $\lambda$, and $exp_\lambda=id:NX_i\rightarrow NX_i\subset P$ for all $i$.
\end{lemma}

\begin{proof}
This can be proved by induction, by using Lemma \ref{extendplumbmetric}. 

We'll use (reverse) induction on $k=\#I$ to prove the following statement: there exist metrics $g_I$ on each $X_I$ satisfying the conditions of Lemma \ref{extendplumbmetric}, that is: 
\begin{enumerate}
\item $\psi_J^I=exp_g:N_{X_I}X_J\rightarrow X_I$ for all $I\subset J\subseteq \mathcal I$ (for $\#I\geq k$);
\item $(\psi_{J}^I)^{*}g_I=\lambda_{J}^I$ on $N_{X_I}X_{J}$ is a bundle metric as per Lemma \ref{bundlemetric}. 
\item In a neighborhood of a higher intersection $X_{J'}$ (where $I\subset J\subset J'$), we have that $\lambda_{J}^I=(\Phi_{J'}^{J})_*(\lambda_{J}^I|_{X_{J'}})$ is the push-forward of the fibrewise metric on $N_{X_I}X_{J'}$ by the (local) bundle isomorphism $N_{X_I}X_{J'}\cong N_{X_I}X_{J}$.
\end{enumerate}

In particular, once we reach $k=1$, the metrics $g_i$ on $X_i$ can be used to obtain a metric $\lambda$ (see Lemma \ref{extendplumbmetric}).
The induction starts with $\#I=k=d$ ($d$ is the depth of the intersection). In this case, $g_I$ can be any metric on $X_I$. 

Inductive step: assume that $g_I$ has been constructed for all $I$ with $\#I\geq k>1$. Let $I'\subset \mathcal I$ have cardinality $\#I'=k-1$. Consider all the $I$'s such that $I'\subset I$, $\#I\geq k$. The plumbing in particular induces plumbing data for the family $\{X_I\}_{I'\subset I}$ in $X_{I'}$. The metrics $g_I$ on $X_I$ satisfy the requirements of Lemma \ref{extendplumbmetric}, by inductive hypothesis. Therefore, we can use Lemma \ref{extendplumbmetric} to build $\tilde g_{I'}$ on $N_{X_{I'}}(\bigcup_I X_I)$. The plumbing maps $\psi_I^{I'}$ can be used to push forward $\tilde g_{I'}$ on $N_{X_{I'}}X_I$ to $g_{I'}$ on (a neighborhood of $X_I$ in) $X_{I'}$ (the fact that this actually induces a metric coherently on a neighborhood of $\bigcup_I X_I$ in $X_{I'}$ follows from the plumbing properties). Now just extend $g_{I'}$ to the whole of $X_{I'}$ (there are no requirements away from $\bigcup_I X_I$). This procedure constructs $g_{I'}$ for any given ${I'}$ when $\#{I'}=k-1$. Such maps satisfy property (1), (2), (3) by construction and so the induction can continue until $k=1$.

\end{proof}

\subsection{Existence of smooth plumbings}

In order to prove that plumbing data always exist, we need one more preliminary lemma:

\begin{lemma}\label{smoothembedlemma}
Given a metric $\mu$ on $M$, and given the data of a plumbing of submanifolds $\{X_i\}_{i\in\mathcal{I}}$ (with plumbing metric $\lambda$) inside a manifold $M$, assume there exists a map
$\nu:N((\bigcup_{\#I=k} X_I)\cup(\bigcup_{i=1}^s X_{J_i}))\rightarrow M$ (for some $J_i\subseteq \mathcal I$, $\#J_i=k-1$ for all $i$)
such that:
\begin{itemize}
\item $\nu(X_i)=X_i$ for all $i$;
\item $\nu^*(mu)=\lambda$.
\end{itemize}
Assume that there exists a subset $J_{s+1}\subseteq \mathcal I$ not equal to any of the $J_i$ for $i\leq s$, such that $\# J_{s+1}=k-1$. Then one can construct a metric $\tilde\mu$ on $M$ and a function $\tilde\nu:N((\bigcup_{\#I=k} X_I)\cup(\bigcup_{i=1}^{s+1} X_{J_i}))\rightarrow M$ such that:
\begin{itemize}
\item $\tilde\nu(X_i)=X_i$ for all $i$;
\item $\tilde\nu^*(\tilde\mu)=\lambda$.
\end{itemize}
\end{lemma}

\begin{proof}
Up to permuting the indexes, we can assume that $J_{s+1}=\{1,\ldots,k-1\}$.
We will have to proceed by induction and construct, for each $0\leq t\leq k-1$, a metric $\mu_t$ on $M$ and a map $\nu_t:N((\bigcup_{\#I=k} X_I)\cup(\bigcup_{i=1}^s X_{J_i})\cup X_t)$ such that:
\begin{itemize}
\item $\nu_t(X_i)=X_i$ for all $i$ in a neighborhood of $(\bigcup_{\#I=k} X_I)\cup(\bigcup_{i=1}^s X_{J_i})$;
\item $\nu_t(X_i)=X_i$ for $0\leq i\leq t$ in a neighborhood of $X_{J_{s+1}}$.
\item $\nu_t^*(\mu_t)=\lambda$.
\end{itemize}
The starting point is $\mu_0=\mu$, $\nu_0=\nu$. The end point is $\tilde\mu:=\mu_{k-1}$ and $\tilde\nu=\nu_{k-1}$. Given $\mu_t,\nu_t$ we can construct a map $\rho:NX_{t+1}\rightarrow M$ by choosing a bundle map $\ell: NX_{t+1}\rightarrow TM|_{X_{t+1}}$ and considering 
$exp_{\mu_t}^{\ell}:NX_{t+1}\rightarrow M$ (see remark at the end of this proof); we make sure to pick $\ell$ so that it agrees with $(d\nu_t)|_{\mathcal{N_{t+1}}}$\footnote{Let $\mathcal M= N((\bigcup_{\#I=k} X_I)\cup(\bigcup_{i=1}^s X_{J_i})\cup X_t)$. Above, $\mathcal N_{t+1} $ refers to $ \mathcal N_{t+1}=N_{\mathcal M}X_{t+1}$, which is naturally identified with a subbundle of $T{\mathcal M}X_{t+1}$ by considering orthogonal directions with respect to $\mu_t$ } in a neighborhood of $(\bigcup_{\#I=k} X_I)$, and $\ell(N_{X_i}X_{t+1})\subset TX_i$. Let 
\[
 \nu_{t+1} =
  \begin{cases} 
      \hfill exp_{\mu_t}^{\ell} \hfill & \text{ on } NX_{t+1} \\
      \hfill \nu_t \hfill & \text{ on } N((\bigcup_{\#I=k} X_I)\cup(\bigcup_{i=1}^{s+1} X_{J_i}))\\
      \hfill id \hfill & \text{ on }X_{t+1}\\
  \end{cases}
\]
and consider the metric
\[
 \nu_{t+1} =
  \begin{cases} \hfill (\nu_{t+1})^*(\lambda) \hfill & \text{ on the image of } \nu_t \\
\hfill \text{} \hfill & \text{any metric on the rest of } M\\
  \end{cases}
\]

We can check that this is well defined: close to $(\bigcup_{\#I=k} X_I)\cup(\bigcup_{i=1}^{s+1} X_{J_i})$, we know that $\mu_t=(\nu_t)_*(\lambda)$ and $\ell$ is the orthogonal splitting due to $\mu_t$, so $exp_{\mu_t}^{\ell}=exp_{\mu_t}=\nu_t\circ exp_\lambda=\nu_t$ because $\lambda$ has the property that $exp_\lambda=id$. So $\nu_{t+1}$ is well defined. Define $\mu_{t+1}=(\nu_{t+1})_*(\lambda)$ Moreover:
\begin{itemize}
\item In a neighborhood of $(\bigcup_{\#I=k} X_I)\cup(\bigcup_{i=1}^{s+1} X_{J_i})$, $\nu_{t+1}=\nu_t$ so by induction $\nu_{t+1}(X_i)=X_i$ for all $i$;
\item in a neighborhood of $X_{J_{s+1}}$, by inductive hypotheses $X_i$ are geodesic with respect to $\mu_t$ for $0\leq i\leq t$; moreover, $\ell(N_{X_i}X_{t+1})\subset TX_i$, therefore for $0\leq i\leq t$, $\mu_{t+1}(X_i)=exp_{\mu_t}^{\ell}(X_i)\subset X_i$; also, by construction $\mu_{t+1}(X_t+1)=X_{t+1}$;
\item $\nu_{t+1}^*(\mu_{t+1})=\lambda$ by construction.
\end{itemize}

\end{proof}

\begin{rmk}
The map $exp_{\mu}:NX\rightarrow M$ for a submanifold $X$ is defined after identifying $NX$ with a subbundle of $TM|_X$; while the usual choice of identification is the one where one identifies $NX$ with the subbundle of $TM$ made of vectors orthogonal to $TX$, any other bundle map $\ell: NX\rightarrow TM|_X$ that yields a splitting $TM\cong TX\oplus \ell(NX)$ can also be used to define a map $exp_{\mu}^{\ell}:NX\rightarrow M$.
\end{rmk}

\begin{prop}\label{smoothembedd}
Given the data of a plumbing of submanifolds $\{X_i\}_{i\in\mathcal{I}}$ inside a manifold $M$, there exists a smooth embedding $F:N(\bigcup_i X_i)\rightarrow M$ such that $F|_{X_i}=id$ for all $X_i$.
\end{prop}

\begin{proof}
By induction, for all $1\leq k \leq \#\mathcal I +2$, for all $0\leq s\leq n=\#\{J\subseteq \mathcal I \text{ such that } \# J=k-1, X_{J_i}\neq \emptyset \}$ we'll construct the following:
\begin{itemize}
\item $\mu_{k,s}$ a metric on $M$;
\item $\nu_{k,s}:N((\bigcup_{\#I=k} X_I)\cup(\bigcup_{i=1}^s X_{J_i}))\rightarrow M$ (where $\#J_i=k-1$);
\end{itemize}
such that:
\begin{itemize}
\item $\nu_{k,s}(X_i)=X_i$ for all $i$ (notice that $N((\bigcup_{\#I=k} X_I)\cup(\bigcup_{i=1}^s X_{J_i}))\subset N(\bigcup_i X_i)$, so $N_{X_j}X_{I}$ is identified with $X_j$ in $N((\bigcup_{\#I=k} X_I)\cup(\bigcup_{i=1}^s X_{J_i}))$);
\item $\nu_{k,s}^*(\mu)=\lambda$.
\end{itemize}

We'll proceed by reverse induction on $k$ and for each fixed $k$ we'll proceed by induction on $s$.
The base case is thus $k=\#\mathcal I+2$, so that $n=0$ and $(\bigcup_{\#I=k} X_I)\cup(\bigcup_{i=1}^s X_{J_i})=\emptyset$. Therefore $\mu_{\#\mathcal I+2, 0}$ can be any metric on $M$.

Let's assume by induction that $\mu_{k,s}$ and $\nu_{k,s}$ have been constructed. If $s=n_k$, then proceed to define $\mu_{k-1,0}:=\mu_{k,n_k}$ and $\nu_{k-1,0}:=\nu_{k,n_k}$. Otherwise, we can find $J_{s+1}$ such that $\# J_{s+1}=k-1$ and $J_{s+1}\neq J_i$ for $i\leq s$. We can then apply Lemma \ref{smoothembedlemma} to build $\mu_{k,s+1}$ and $\nu_{k,s+1}$. This concludes the induction.

The end point of the induction produces the map $\nu_{1,0}:N(\bigcup_i X_i)\rightarrow M$ such that $\nu_{1,0}(X_i)=X_i$ for all $i$ and this concludes the proof.

\end{proof}

\begin{rmk}
It is tempting to try and write a much simpler proof of the above proposition by interpolating metrics that can be constructed locally; unfortunately this doesn't work because interpolation of metrics in general doesn't preserve geodesics.
In other words, if $\gamma$ is a geodesic in $M$ with respect to both $\mu_1, \mu_2$, and if $\mu$ is some interpolation of $\mu_1$ and $\mu_2$, then $\gamma$ need not be a geodesic for $\mu$. Similarly, if $X$ is a totally geodesic submanifold with respect to different metrics, the property may be lost when interpolating with respect to a partition of unity.
\end{rmk}

\begin{ex}
An easy counterexample is the following: let $M=\R^2$, consider the standard metric $\mu_1=dx^2+dy^2$ and a  scaled version $\mu_2=\frac{1}{10}\mu_1$. Consider a partition of unity $t_1,t_2:\R^2\rightarrow \R$ with respect to the two open sets: $\mathcal U_1=\{(x,y)\in \R^2| x^2+y^2<1\}$ and $\mathcal U_2=\{(x,y)\in \R^2| x^2+y^2>\frac{1}{2}\}$, i.e. $t_1+t_2=1$ and $t_i=0$ outside of $\mathcal U_i$. Let $\mu=t_1\mu_1+t_2\mu_2$.

Consider $\gamma=[-1,1]\times \{0\}\subset \R^2$. This is a geodesic for both $\mu_1,\mu_2$ but not for $\mu$. 
\end{ex}

We can now prove that plumbings exist:
\begin{prop}\label{smoothplumbingsexist}
Given a finite family of smooth submanifolds $\{X_i\}_{i\in\mathcal I}$ of $M$, there exist plumbing data. Therefore one can build a plumbing, and such plumbing is diffeomorphic to a neighborhood of $\bigcup_i X_i$ in $M$. 
\end{prop}

\begin{proof} 
We will prove, by induction, that for any $0\leq k\leq n=\#\mathcal I$ there exist systems of tubular neighborhoods $\psi_I^{I'}$ for $I'\subset I$ and $\# I'\geq k$ and compatible systems of bundle isomorphisms $(\Phi_I^{I'})_J:N_{X_J}X_I\rightarrow N_{X_J}X_{I'}$ for $J\subset I'\subset I$, $\#J\geq k$. The starting point is when $k$ is the deepest level of intersection. In that case, the system of tubular neighborhoods is given by any choice of tubular neighborhood, and there are no bundle isomorphisms. By induction, let's assume the existence of a system of tubular neighborhoods and bundle isomorphisms for a fixed $k$. Let $H\subset J$ where $\#J=k, \#H=k-1$. Consider $\rho_I^{I'}$ for some $I'\subset I$, $\#I'=k$, $\#I=k+1$ and pick a bundle isomorphism $(\tilde\Phi_I^{I'})_H:\rho^*(N_{X_H}X_{I'}|_I)\rightarrow N_{X_H}X_{I'}|_{\mathcal U_I^{I'}}$ extending the isomorphisms $(\tilde\Phi_I^{I'})_J:\rho^*(N_{X_J}X_{I'}|_I)\rightarrow N_{X_J}X_{I'}|_{\mathcal U_I^{I'}}$ defined for all $J\supset H$. This yields bundle isomorphisms $(\Phi_I^{I'})_H:N_{X_H}X_I\rightarrow N_{X_H}X_{I'}$, compatible with the rest of the data. This allows us to define the plumbing of the higher intersections inside $X_H$, and find an embedding $\phi:N(\bigcup_{J\supset H} X_J)\rightarrow X_H$ (this is possible by Proposition \ref{smoothembedd}). Such an embedding yields a system of tubular neighborhoods $\psi_I^{H}$ compatible with the previous data. This concludes the induction. When $k=0$, we obtain plumbing data for $X_i$'s in $M$.
\end{proof}
\begin{cor}\label{totallygeodesicmetric}
Given transverse submanifolds $\{X_i\}_{i\in \mathcal I}$ in $M$, one can find a metric $\mu$ on $M$ such that each $X_i$ is geodesic and, for all $i,j$, the image of ${exp_{\mu}}|_{N_{X_j} X_i}:N_{X_j}X_i \rightarrow M$ in a neighborhood of $X_{ij}$ is contained in $X_j$.
\end{cor}

\begin{rmk}
The corollary has no claim of originality. In \cite{Milnor} for instance, a nice proof of this fact is presented for the intersection of two manifolds (the proof is attributed to E. Feldman). Such proof only applies to n=2 submanifolds, but it can easily be extended to general n by induction.

The reason to present this result as corollary, is that the result of Proposition \ref{smoothplumbingsexist} is slightly stronger than the corollary, and will be needed in the rest of the discussion.
\end{rmk}

\begin{proof}
Consider an embedding $\rho$ of the plumbing inside $M$. Then let $\mu=\nu_*(\lambda)$ be the pushforward of the plumbing metric. Then $\mu$ can be extended to a metric on all of $M$ which satisfies the requirements. 
\end{proof}

\section{Rigid plumbing: definition and properties}

In order to pursue an analogue of Theorem \ref{snt} one needs to:
\begin{enumerate}[i)]
\item define a standard symplectic model $(N(\bigcup_i X_i),\omega)$ for a family of crossing submanifolds (an analogue of the symplectic normal bundle for a single submanifold);
\item find an embedding $f: (N(\bigcup_i X_i),\omega)\rightarrow (M,\omega_M)$ such that $f^*(\omega_M)=\omega$ along $\bigcup_i X_i$. 
\end{enumerate}

Part i) is achieved in Section \ref{symplecticformonplum}. 

Part ii) requires that the plumbing embeds in a more rigid way, which is what we explore in this section.

In view of trying to embed the normal bundles into $M$ symplectically, we need to keep track of the orthogonal directions in $M$.
In terms of differential geometry, this can be done by introducing maps $g_i:N_M(X_i)\rightarrow TM|_{X_i}$, left inverses of the quotient map $TM|_{X_i}\rightarrow TM|_{X_i}/TX_i$, so that they yield a splitting $ TM|_{X_i}\cong  TM|_{X_i}\oplus g_i(N_M(X_i))$. We will then look for embeddings $f_i:NX_i\rightarrow M$ such that the orthogonal direction in $NX_i$ gets mapped to $TM$ via $g_i$; this isn't hard to do. It becomes hard if we require such embeddings $f_i$ to also agree on the plumbing. Such $f_i$'s can't be built in general for any choice of plumbing and $g_i$'s, unless we make some more compatibility assumptions.

\begin{defn}\label{compatiblesplitting}
A \textbf{splitting} on a finite family of transverse submanifolds $\{X_i\}_{i\in\mathcal I}$ is a collection of maps $\{g_i:N_M(X_i)\rightarrow TM|_{X_i}\}_{i\in\mathcal I}$, each of them yielding a splitting $ TM|_{X_i}\cong  TM|_{X_i}\oplus N_M (X_i)$.

A splitting is \textbf{compatible} with the data of a plumbing and an embedding 

$\nu:N(\bigcup_i X_i)\rightarrow M$ if:\begin{itemize}
\item $g_i|_{N_{X_i}X_{ij}}=d_{X_{ij}}\psi_{ij}^i:N_{X_i}X_{ij}\rightarrow TX_i\subset TM|_{X_i}$ for all $i$;
\item ${d_{X_j}d_{X_i}\nu }^{-1} \circ d_{N_{X_j}X_{ij}} (g_i\circ\Phi_{ij}^i)={d_{X_i}d_{X_j}\nu }^{-1} \circ d_{N_{X_i}X_{ij}} (g_j\circ\Phi_{ij}^j)$:

\ $T_{X_j}(T_{X_i}NX_{ij})\rightarrow T_{X_j}(T_{X_i}NX_{ij})$ 
 \ \ \ for all $i,j$ (see remark).

A splitting is \textbf{compatible} with the data of a plumbing if it is compatible with one, and therefore every, embedding of the plumbing (see Lemma \ref{compatibilityinvariance}).
 
\end{itemize}
\end{defn}

\begin{rmk}
The definition above depends on an identification
$T_{X_i}(T_{X_j}NX_{ij})\leftrightarrow T_{X_j}(T_{X_i}NX_{ij})$ which is natural since $NX_{ij}\cong (NX_i\oplus NX_j)|_{X_{ij}}$. 
\end{rmk}

\begin{lemma}\label{compatibilityinvariance}
Consider a splitting on $\{X_i\}_i\in \mathcal I$, the data of a plumbing, and two different embeddings $\nu, \mu: N(\bigcup_i X_i)\rightarrow M$. If the splitting is compatible with $\nu$, then it is compatible with $\mu$.
\end{lemma}

\begin{proof}
Of course the first compatibility condition is independent from $\mu, \nu$.
For the second condition, observe that $\mu=\nu\circ \tilde\mu$ where $\tilde\mu=\nu^{-1}\circ \mu: N(\bigcup_i X_i)\rightarrow N(\bigcup_i X_i)$. We are assuming that ${d_{X_j}d_{X_i}\nu }^{-1} \circ d_{N_{X_j}X_{ij}} (g_i\circ\Phi_{ij}^i)={d_{X_i}d_{X_j}\nu }^{-1} \circ d_{N_{X_i}X_{ij}} (g_j\circ\Phi_{ij}^j)$, then 
$${d_{X_j}d_{X_i}\mu }^{-1} \circ d_{N_{X_j}X_{ij}} (g_i\circ\Phi_{ij}^i)={d_{X_j}d_{X_i}(\nu\circ \tilde\mu) }^{-1} \circ d_{N_{X_j}X_{ij}} (g_i\circ\Phi_{ij}^i)=$$
$$ {d_{X_j}d_{X_i}\tilde\mu }^{-1}{d_{X_j}d_{X_i}\nu }^{-1} \circ d_{N_{X_j}X_{ij}} (g_i\circ\Phi_{ij}^i)$$
and similarly 
$${d_{X_i}d_{X_j}\mu }^{-1} \circ d_{N_{X_i}X_{ij}} (g_j\circ\Phi_{ij}^j)={d_{X_i}d_{X_j}\tilde\mu }^{-1}\circ {d_{X_i}d_{X_j}\nu }^{-1} \circ d_{N_{X_i}X_{ij}} (g_j\circ\Phi_{ij}^j).$$
Since it is always true that $d_{X_i}d_{X_j}\tilde\mu=d_{X_j}d_{X_i}\tilde\mu$ (because this can be checked locally, which is equivalent to reducing to Euclidean space; in Euclidean space, this is Schwarz's) then 
$${d_{X_j}d_{X_i}\nu }^{-1} \circ d_{N_{X_j}X_{ij}} (g_i\circ\Phi_{ij}^i)={d_{X_i}d_{X_j}\nu }^{-1} \circ d_{N_{X_i}X_{ij}} (g_j\circ\Phi_{ij}^j).$$
\end{proof}

\begin{comment}
Insert Lemma: The definition of compatible splittings is independent of the chosen system of coordinates.

\begin{ex}
\fixme{insert the examples of compatible splitting and non compatible}
\end{ex}

\end{comment}

\subsection{Rigid plumbing embedding}
Proposition \ref{smoothembedd} can be strengthened to a more rigid setup:

\begin{prop}\label{rigidembedding}
Given the following:
\begin{itemize}
\item a family $\{X_i\}_{i\in\mathcal I}$ of transversely intersecting submanifolds in a smooth manifold $M$;
\item the data of a plumbing $N(\cup_i X_i)$ as in Definition \ref{plumdata};
\item a splitting for the family: $\{g_i:N_M(X_i)\rightarrow TM|_{X_i}\}_{i\in\mathcal I}$ compatible with the plumbing as in Definition \ref{compatiblesplitting}. 
\end{itemize}
Then there exist maps
$$F^i:NX_i\rightarrow M$$
such that:
\begin{enumerate}[(i)]
\item $F^i|_{X_i}=id$;
\item $d_{X_i}F^i=g_i+id:T(NX_i)|_{X_i}\cong NX_i\oplus TX_i\rightarrow TM|_{X_i}$;
\item such $F^i$'s agree with respect to $\sim$.
\end{enumerate}
This last property implies that they descend to a map on the quotient $$F:N(\cup_i X_i)\rightarrow M.$$ 
\end{prop}

\begin{rmk}
Without $(ii)$, that is, the requirement on the derivative of $F^i$, this would be the same as Proposition \ref{smoothembedd}.
Condition $(ii)$ is important in view of the next section: it is needed in order to refine the construction so that the maps $F^i$'s become symplectomorphisms.
Notice that a symplectic form on $M$ yields a natural choice of $g_i:N_M(X_i)\cong TM|_{X_i}^{\perp}\rightarrow TM|_{X_i}$ given by an orthogonal splitting, for all $i$.
\end{rmk}

Before proving the statement, let's establish some lemmas.

\begin{lemma}\label{straightening}
Given a fibrewise linear bundle isomorphism $d\eta:T_X(NX)\rightarrow T_XM$, and a map $g:N\rightarrow T_XM$ yieding a splitting $T_XM\cong\tilde NX\oplus TX$, there exists a diffeomorphism $\tilde\eta:NX\rightarrow NX$ such that $d\eta \circ d(\tilde\eta)^{-1}=g\times id:T_X(NX)\rightarrow T_XM$.
\end{lemma}

\begin{proof}
Notice that $T_X(NX)\cong NX\oplus TX$ canonically, so that we get an isomorphism $g\times id:T_X(NX)\rightarrow TM$. So we get an isomorphism $(g\times id)^{-1}\circ d\eta:T_X(NX)\rightarrow T_X(NX)$. We can integrate this isomorphism to a diffeomorphism $\tilde\eta:NX \rightarrow NX$. 

Let's build $\tilde\eta$ explicitely as follows: start with a standard metric $\lambda$ on $NX$ (as in Lemma \ref{bundlemetric}); consider the subbundle $NX$ of $T_X(NX)$, and its image, the subbundle $\ell=(g\times id)^{-1}\circ d\eta (NX)\subset T_X(NX)$. Consider the exponential map $\exp_{\lambda}^{\ell}:N_{NX}X\rightarrow NX$. Since $N_{NX}X\cong NX$ canonically, let $\tilde\eta:=\exp_{\lambda}^{\ell}:NX\rightarrow NX$.
\end{proof}

\begin{lemma}\label{straighteningonplumbing}
Let $N(\bigcup_i X_i)$ be a plumbing, with plumbing metric $\lambda$; let $\{g_i:NX_i\rightarrow T_{X_i}M\}_{i\in\mathcal I}$ be a splitting compatible with the plumbing, as in Definition \ref{compatiblesplitting}. Let $\eta:N(\bigcup_iX_i)\rightarrow M$ be an embedding as built in Proposition \ref{smoothembedd}.

Assume moreover that, for some $j$, $d_{N_{X_i}X_{ij}} (g_j\circ\Phi_{ij}^j)=id:T_{X_{ij}M\rightarrow T_{X_{ij}M}}$ and that $d_{X_j}\eta=g_j\times id:NX_j\times TX_j\rightarrow T_{X_j}M$.
Apply the construction of Lemma \ref{straightening} to $d_{X_i}\eta$.
Then the linearization of the map $\tilde\eta$ along $X_j$ is the identity: $d_{X_j}\tilde\eta=id$.
\end{lemma}

\begin{rmk}
Notice that we can freely use the notation $d_{X_j}\tilde \eta$ when $\tilde\eta:NX_i\rightarrow NX_i$, since $X_j$ has been identified with $N_{X_j}X_i\subset NX_i$.
\end{rmk}

\begin{proof}
This is almost a tautology. By compatibility assumptions, $d_{N_{X_j}X_{ij}} (g_i\circ\Phi_{ij}^i)=$\linebreak $ d_{N_{X_i}X_{ij}} (g_j\circ\Phi_{ij}^j):T_{X_{ij}}M\rightarrow T_{X_{ij}}M$. Let's build $\tilde\eta$ as in Lemma \ref{straightening}, with respect to the plumbing metric $\lambda$. Then the map $(g_i\times id)\circ d_{X_i}\eta:NX_i\rightarrow NX_i$ along $X_j$ has derivative 
$$d_{X_j}(g_i\times id)\circ d_{X_j}d_{X_i}\eta=d_{X_i}(g_j\times id)\circ d_{X_i}d_{X_j}\eta=id: N_{X_j}NX_i\rightarrow N_{X_j}NX_i$$
since: $d_{X_j}(g_i\times id)=d_{X_i}(g_j\times id)$ by assumption on compatibility; $d_{X_j}d_{X_i}\eta=d_{X_i}d_{X_j}\eta$ is true for any $\eta$; and $d_{X_i}(g_j\times id)\circ d_{X_i}d_{X_j}\eta=id$ is the assumption on $\eta$.
\end{proof}

\begin{lemma}\label{extendstraightening}
The map $\tilde \eta$ in Lemma \ref{straighteningonplumbing} can be extended to a map $\tilde{\tilde\eta}:N(\bigcup_iX_i)\rightarrow N(\bigcup_iX_i)$ such that $\tilde{\tilde\eta}|_{\bigcup_i X_i}=id $, $d\eta\circ d_{X_i} (\tilde{\tilde\eta})^{-1}=g_i\times id$ and $d_{X_j}\tilde{\tilde\eta}=id$ for all $X_j$ satisfying the hypotheses of \ref{straighteningonplumbing}. 
\end{lemma}

\begin{proof}
We can do this by explicitly interpolating the map $\tilde \eta$ defined on a neighborhood of $X_i$ with the identity map. Just consider local coordinates on the plumbing: on $NX_I$, $I=\{k_1,\ldots,k_m\}$, we have coordinates $(x,v_{k_1},\ldots,v_{k_m})$ where $x\in X_{I}, v_k\in NX_k|_{X_I}$. Consider the function $\tilde\eta:NX_i\rightarrow NX_i$ defined on a neighborhood $\mathcal V$ of $X_i$ on the plumbing. Consider a bump function $\beta(|v_i|)$ (the norm of $v_i$ is relative to the plumbing metric 
$\lambda$) and neighborhoods $\mathcal U_1\subset\mathcal U_2\subset \mathcal V$ of $X_i$ such that $\beta=0$ on $\mathcal U_1$ and $\beta=1$ on the complement of $\mathcal U_2$. Define $\tilde{\tilde\eta}:\mathcal V\rightarrow\mathcal V$ as $\tilde{\tilde\eta}=(1-\beta)\tilde\eta+\beta\cdot id$. Such function extends to $\tilde{\tilde\eta}:N(\bigcup_iX_i)\rightarrow N(\bigcup_iX_i)$.  As long as $d_{X_j}\tilde\eta =id$, also $d_{X_j} \tilde{\tilde\eta}=id$.
\end{proof}

\begin{proof}[Proof of Proposition \ref{rigidembedding}]
To begin, we can use Proposition \ref{smoothembedd} to construct a smooth embedding, i.e. we get maps $f^i:NX_i\rightarrow M$ satisfying properties (i) and (iii) but not necessarily (ii). From property (iii), this is the same as having an embedding $f:N(\bigcup_i X_i)\rightarrow M$.

Now we'll construct a map $\eta: N(\bigcup_i X_i)\rightarrow N(\bigcup_iX_i)$ in such a way that the composition $F=f\circ\eta$ will give the required rigid embedding $F:N(\bigcup_iX_i)\rightarrow M$.
The idea is to build $\eta$ as a composition of functions $\eta_1, \ldots,\eta_n$ that ``straighten'' $f$ along each $X_i$.

Let's start by applying Lemma \ref{straightening} and Lemma \ref{extendstraightening} to the linear bundle isomorphisms $df_1:T(NX_1)\rightarrow T_{X_1}M$ and  $g_1\times id:NX_1\oplus TX_1\rightarrow TM$, to build a map $\widetilde{\widetilde f_1}: N(\bigcup_i X_i)\rightarrow N(\bigcup_iX_i)$ such that $df_1\circ d_{X_1}\widetilde{\widetilde f_1}=g_1\times id$. Let
$\eta_1:= \widetilde{\widetilde f_1}$.

We can proceed by induction: let $\eta_k:N(\bigcup_i X_i)\rightarrow N(\bigcup_iX_i)$ be the function obtained by applying Lemma \ref{straightening} and Lemma \ref{extendstraightening} to $d_{X_k}(f\circ\eta_1\ldots\circ\eta_{k-1})$. Such $\eta_k$ has the property that $d_{X_k}(f\circ\eta_1\ldots\circ\eta_{k-1})\circ d_{X_k}\eta_k=g_k\times id$. Moreover, since $d_{X_{k-1}}(f\circ\eta_1\ldots\circ\eta_{k-1})=d_{X_{k-1}}(f\circ\eta_1\ldots\circ\eta_{k-2})\circ d_{X_{k-1}}\eta_{k-1}=g_{k-1}\circ id$, by Lemma \ref{straighteningonplumbing} we have that $d_{X_{k-1}}\eta_k=id$ and similarly $d_{X_i}\eta_k=id$ for all $i\leq k-1$.

Now consider $F:N(\bigcup_i X_i)\rightarrow M$ given by $F=f\circ\eta_1\circ\ldots\circ\eta_n$. This function on the plumbing restricts to functions $F_i:NX_i\rightarrow M$ satisfying (i) and (iii). Let's check that they also satisfy (ii).
By definition, $d_{X_k}F_k=d_{X_k} f_k\circ d_{X_k} \eta_1\circ\ldots\circ d_{X_k}\eta_n$. By construction, $d_{X_k}\eta_i=id$ for all $i\geq k+1$. Therefore, $d_{X_k}F_k=d_{X_k} f_k\circ d_{X_k} \eta_1\circ\ldots\circ d_{X_k}\eta_k=g_k\times id$ by construction of $\eta_k$.
\end{proof}

\begin{thm}\label{plumbingsexist}
Let $\{X_i\}_{i\in\mathcal{I}}$ be a finite set of transverse submanifolds.
Given a splitting $g_i:NX_i\rightarrow TM$ as in Definition \ref{compatiblesplitting}, assume that $g_i|_{N_{X_j}X_i\subset NX_i}:N_{X_j}X_i\rightarrow TX_j\subset TM$ for all $i,j$. One can find:

\begin{itemize}
\item plumbing data to build a plumbing $N(\bigcup_i X_i)$ compatible with the splitting;
\item an embedding $F: N(\bigcup_i X_i)\rightarrow M$ such that $d_{X_i}F=(g_i,id):T(NX_i)\rightarrow TM$.
\end{itemize} 
\end{thm}

\begin{proof}
This is done by induction in the following way: for any $1\leq k\leq n$, for all $I$ such that $\# I=n-k$, we build a $k$-plumbing of $\{X_{I\cup\{i\}}\}_{i \notin I}$ in $X_{I}$ compatible with the induced splittings $g_i|_{X_I}:NX_i|_{X_I}\rightarrow T_{X_I}M$. 
For $k=1$, let's start by constructing a smooth plumbing of $\bigcup_i X_i$ and embedding it in $M$. This induces a metric $\lambda$ on $M$. Let $\psi_\mathcal I^{\mathcal I \backslash \{i\}}=exp_\lambda N_{X_{\mathcal I \backslash \{i\}}X_\mathcal I}$ for all $i$. These maps are then automatically compatible with the splitting. 

For the inductive step, let's assume that all $\{\psi_I^{I'}\}_{I'\subset I\subset \mathcal I, \#I'\geq n-k}$ have been constructed such that they are compatible with the restrictions of $g_i$ to $N_{X_{I'}}X_I$ for all $I'\subset I\subset \mathcal I, \#I'\geq n-k$. 
Pick $J$ such that $\#J=n-k-1$, let's construct plumbing data for $\bigcup_{\#I=n-k, I\supset J} X_I$ in $X_J$ and an embedding $\nu_J$ of such plumbing. A priori this data is not compatible with the maps $g_i$, so we will use the map $\nu_J$ to construct maps $\Phi_I^{J}$ for $\#I=n-k, I\supset J$ such that the new plumbing is compatible with the maps $g_i$. 
On a neighborhood of $X_{J\cup\{i, j\}}$, 
the condition
$${d_{X_{J\cup\{j\}}}d_{X_{J\cup\{i\}}}\nu_J }^{-1} \circ d_{N_{X_{J\cup\{j\}}}X_{J\cup\{i, j\}}} (g_i\circ\Phi_{J\cup\{i, j\}}^{J\cup\{i\}})=$$ 

$${d_{X_{J\cup\{i\}}}d_{X_{J\cup\{j\}}}\nu_J }^{-1} \circ d_{N_{X_{J\cup\{i\}}}X_{J\cup\{i, j\}}} (g_j\circ\Phi_{J\cup\{i, j\}}^{J\cup\{j\}})$$
can be seen as a condition on $d_{N_{X_{J\cup\{i\}}}X_{J\cup\{i, j\}}}\Phi_{J\cup\{i, j\}}^{J\cup\{j\}}$, given $\nu_J, g_i, g_j, \Phi_{J\cup\{i, j\}}^{J\cup\{j\}}$. Therefore we can fix any $\Phi_{J\cup\{i, j\}}^{J\cup\{i\}}$ and then construct $\Phi_{J\cup\{i, j\}}^{J\cup\{j\}}$ which satisfies the requirement.

At this point we invoke Proposition \ref{rigidembedding} to build embeddings $ \{\psi_I^{J}\}_{J\subset I\subset \mathcal I}$ such that $d_{X_{J\cup\{i\}}} \psi_{J\cup\{i\}}^J=g_i|_{N_{X_J}{X_{J\cup\{i\}}}}$; this concludes the induction.

For $k=n$, this yields the $n$-plumbing; Proposition \ref{rigidembedding} then yields an embedding $F: N(\bigcup_i X_i)\rightarrow M$ such that $d_{X_i}F=(g_i,id):T(NX_i)\rightarrow TM$.
\end{proof}

\section{Symplectic form on plumbing}\label{symplecticformonplum}
In a symplectic setting, we want a model of plumbing that carries a symplectic form and can be symplectically embedded in $M$. 
As discussed before, a choice of connection on a symplectic vector bundle yields a symplectic form in a neighborhood of the zero section. 
Therefore, the natural way to induce a symplectic form on 
$N(\bigcup_iX_i)$ is to pick connections $\{\alpha_i\}$ on each $NX_i$, then show that the induced symplectic forms $\omega_{\alpha_i}$ on $NX_i$ naturally agree inside $N(\bigcup_i X_i)$. 
This is not true for any arbitrary choice of connections.

This section shows how to pick the right connections $\{\alpha_i\}$ on each $NX_i$, in such a way that the forms $\omega_{\alpha_i}$ agree when gluing the normal bundles together.

\begin{defn}\label{compatibleconn}
A connection $\alpha_i$ on $NX_i$ is \textbf{compatible} with the plumbing data when, for any $I$ such that $i\in I$, $\alpha_i$ is the pushforward, by $\Phi_I^i$, of a connection on $NX_i|_{X_I}$, that is, $$\alpha_i=(\Phi_I^i)_*(\alpha_i|_{{X_I}}).$$
\end{defn}

In particular this implies that the curvature of $\alpha_i$ vanishes in certain directions close to $X_I$.

To have full compatibility with the symplectic structure, we need one more definition: 
\begin{defn}\label{symplumdata}
If $M$ is symplectic and $X_i$ are symplectic submanifolds, a system of bundle isomorphisms is \textbf{symplectic} when the isomorphisms $\Phi_I^J$ are all isomorphisms of symplectic bundles.

A system of tubular neighborhoods is \textbf{symplectic} when, for all $I$ such that $\#I\geq 2$, there exist connections $\alpha_I=\oplus_{i\in I} \alpha_i|_{X_I}$ such that $\psi_I^{I'}:(N_{X_{I'}}X_I,\omega_{\alpha_I})\rightarrow (X_{I'}, \omega_{I'})$ is a symplectomorphism.

A connection $\alpha_i$ on $NX_i$ is \textbf{compatible with the symplectic plumbing data} when $\alpha_i$ is a symplectic connection and $\alpha_i=(\Phi_I^i)_*(\alpha_i|_{X_I})$ for all $I$ containing $i$.
\end{defn}

\begin{lemma}\label{nconnections}
Let $\{X_i\}_{i\in\mathcal{I}}, \{\psi_I^{I'}\},\{\Phi_I^{I'}\} $ be the data of a plumbing as in Definition \ref{plumdata}. Then there exist connections $\{\alpha_i\}$ on each $NX_i$ which are compatible with the plumbing. If the plumbing data is symplectic as in Definition \ref{symplumdata}, then the connections can be chosen to be compatible with the symplectic data.
\end{lemma} 

\begin{proof}

Let's first prove this for the non-symplectic plumbing.
Let $\#\mathcal{I}=n$. 
The result is best proved by induction on $k$, by building connections $\alpha_{I}=\oplus_{i\in I} \alpha_i |_{X_I}$ on $NX_I$ for $\# I=n-k$, such that, for all $J\supset I$, $\alpha_I=(\Phi_J^I)_*(\alpha_I |_{X_J})$. For $k=n-1$ this yields the result.
For $k=0$, there is no compatibility restriction so any choice of $\alpha_I$ is good.

For the inductive step, assume $\alpha_J$ have been built for all $J$ with $\# J > n-k$. Let $\#I=n-k$ and let $\alpha_I=(\Phi_J^I)_*(\alpha_J)$ on a neighborhood of $X_J$ for $\#J=n-k+1$. This way we have a well defined connection on a neighborhood of $\bigcup_{\#J=n-k+1} X_J\subset X_I$. We can use a partition of unity to extend $\alpha_I$ to all of $X_I$.

If the plumbing is symplectic, the conditions coming from Definitions \ref{plumdata} and \ref{symplumdata} imply that the connections have already been built up to the level $\# J=2$. We can use the same procedure described above to extend $\alpha_i=(\Phi_J^i)_*(\alpha_J)$ to a symplectic connection on all of $NX_i$.
\end{proof}

Once we have symplectic plumbing data and compatible connections, we can decorate the plumbing with a symplectic form, thanks to the next lemma.

\begin{lemma}\label{sumofbundles}
Let $\psi_X:(N_X(X\cap Y),\omega_{\alpha_Y|_{X\cap Y}})\rightarrow (X,\omega_X)$, $\psi_Y:(N_Y(X\cap Y),\omega_{\alpha_X|_{X\cap Y}})\rightarrow (Y,\omega_Y)$ be symplectic embeddings. If $X\perp^\omega Y$, and if $\alpha_X$, $\alpha_Y$ are symplectic connections which are compatible with the data as in Definition \ref{compatibleconn}, the plumbing $N_M(X\cup Y)$ inherits a natural symplectic structure descending from $(NX,\omega_{\alpha_X})$, $(NY,\omega_{\alpha_Y})$, $(N(X\cap Y),\omega_{\alpha_Y|_{X\cap Y}\oplus\alpha_X|_{X\cap Y}})$. 
\end{lemma}

\begin{proof}
We need to check that $\omega_{\alpha_X}$ and $\omega_{\alpha_Y}$ agree with respect to $\sim$, therefore inducing a symplectic form on the quotient. Let's consider the embedding $\xi_X: N_M(X\cap Y)\rightarrow NX$ where $N_M(X\cap Y)=N_X(X\cap Y)\oplus N_Y(X\cap Y)$ comes with the connection ${\hat\alpha_X\oplus\hat\alpha_Y}$ and the induced symplectic form. Then showing that $\xi_X^*(\omega_{\alpha_X})=\omega_{\hat\alpha_X\oplus\hat\alpha_Y}$ is enough to prove the statement. Recall that $\rho_X,\rho_Y$ indicate the fiberwise hamiltonian action on $NX,NY$ respectively (and the corresponding induced action on $N(X\cap Y)$).

We are looking for $\xi_X^*(\omega_{\alpha_X})=\xi_X^*(\pi^*(\omega_X)+\rho_X F_{\alpha_X}+\omega_{\C^k})$, to compare it to
 $$\omega_{\hat\alpha_X\oplus\hat\alpha_Y}= \pi^*(\omega_{X\cap Y})+\rho_Y F_{\hat\alpha_X}+\rho_X F_{\hat\alpha_Y}+\omega_{F}$$
 (here $\omega_F$ is the symplectic structure on a fibre of $N(X\cap Y)$).

We already know that $\psi_X^*(\omega_X)=\pi^*(\omega_{X\cap Y})+\rho_Y F_{\hat\alpha_X}+\omega_\C^h$ 

from which we get $\xi_X^*(\pi_X^*(\omega_X))=
\psi_X^*(\omega_X)=\pi^*(\omega_{X\cap Y})+\rho_Y F_{\hat\alpha_X}+\omega_{\C^h,X}$ (here $\omega_{\C^h,X}$ is the symplectic form on the fibres of $N(X\cap Y)$ in the direction of $X$).

Also it is easy to see that $\xi_X^*(\omega_{\C})=\omega_{\C^k,Y}$ (respectively the symplectic form on the vertical fibres of $NX$ and the symplectic form  on the fibres of $N(X\cap Y)$ in the direction orthogonal to $X$).

As for $\xi_X^*(\rho_X F_{\alpha_X})$, recall that it vanishes over $N_{X,z}(X\cap Y)$ by assumption, therefore 

$\rho_X F_{\alpha_X}=\rho_X F_{\alpha_X}|_{\pi^{-1}(X\cap Y)}\circ \pi_{\text{vert}}$ which also implies 

$\xi_X^*(\rho_X F_{\alpha_X})=\rho_X F_{\hat\alpha_Y}$.

Therefore
$$\xi_X^*(\omega_{\alpha_X})=\xi_X^*(\pi^*(\omega_X)+\rho_X F_{\alpha_X}+\omega_{\C^k})=$$
$$\pi^*(\omega_{X\cap Y})+\rho_Y F_{\hat\alpha_X}+\omega_{\C^k,X}+\rho_XF_{\hat\alpha_Y}+\omega_{\C^h,Y}=\omega_{\hat\alpha_X\oplus\hat\alpha_Y}.$$
Where the very last equality uses the fact that $X\perp^\omega Y$, therefore the symplectic structure on a fibre of $N(X\cap Y)$ is $\omega_F=\omega_{\C^k,X}+\omega_{\C^h,Y}$.
\end{proof}

\begin{prop}\label{symplum}
Let $\{\Phi_I^{I'},\psi_I^{I'}\}$ be a set of symplectic plumbing data for $\{X_i\}_{i\in\mathcal I}$ as in Definitions \ref{plumdata}, \ref{symplumdata}. Let $\alpha_i$ be a compatible connection on $NX_i$ for each $i$, as in Definition \ref{compatibleconn}. Then, the plumbing inherits a symplectic form $\omega_{\mathcal I}$.
\end{prop}

\begin{proof}
Let's construct the symplectic forms $\omega_{\alpha_i}$ on each $NX_i$.
Following Lemma \ref{sumofbundles}, $\omega_{\alpha_i}=\omega_{\alpha_j}$ on the plumbing (wherever they are defined). This means that all of the $\omega_{\alpha_i}$ glue onto a symplectic form $\omega_{\mathcal I}$ defined on the total space of the plumbing.  
\end{proof}

At this point we have enough information to state the following theorem. 

\begin{thm}[Neighborhood theorem for orthogonal submanifolds]\label{main}
Let $\{X_i\}_{i\in\mathcal I}$ be a finite set of symplectic submanifolds of $(M,\omega)$ with orthogonal intersections. Then there exist:
\begin{itemize}
\item symplectic plumbing data yielding a plumbing $N(\bigcup_i X_i)$;
\item compatible symplectic connections which induce a symplectic form $\omega_{\mathcal I}$ on $N(\bigcup_i X_i)$;
\item a neighborhood of $\bigcup_i X_i$ in $(M,\omega)$ symplectomorphic to a neighborhood of $\bigcup_i X_i$ inside of 
$(N(\bigcup_i X_i), \omega_\mathcal I)$ (such symplectomorphism is relative to $\bigcup_i X_i$).
\end{itemize}
\end{thm}

Just like we did for Propositions \ref{smoothplumbingsexist} and Theorem \ref{plumbingsexist}, we wish to prove this theorem by induction. The main step in the induction is provided by the following proposition, a symplectic analog of Lemma \ref{smoothembedd}.

\begin{prop}\label{symplembedding}
Let $\{X_i\}_{i\in\mathcal I}$ be a finite set of symplectic submanifolds of $(M,\omega)$ with orthogonal intersections. Let $\{\psi_I^{I'}\}_{I'\subset I\subseteq \mathcal I}, \{\Phi_I^{I'}\}_{I'\subset I\subseteq \mathcal I, \# I'\geq 2}$ be (partial) symplectic plumbing data for $\{X_i\}_i$. Then there exist:
\begin{itemize}
\item maps $\Phi_I^i:NX_I\rightarrow NX_i$ completing the data to that of a symplectic plumbing;
\item compatible symplectic connections which induce a symplectic form $\omega_{\mathcal I}$ on the symplectic plumbing;
\item a symplectic embedding $f:(N(\bigcup_iX_i),\omega_{\mathcal I})\rightarrow (M,\omega)$, defined in a neighborhood of $\bigcup_i X_i$, such that $F|_{X_i}=id$.
\end{itemize}
\end{prop}

\begin{proof}
Notice that the symplectic form $\omega$ induces orthogonal splittings of the tangent bundle along any $X_i$, i.e. we get standard maps $g_i:NX_i\rightarrow TM|_{X_i}$ (defined as $g_i(v)=w$ where $\pi(w)=v$ and $w\perp^{\omega}TX_i$). Since the plumbing data is symplectic, it follows that the splitting obtained through the symplectic form must be compatible (as described in Definition \ref{compatiblesplitting}) with the partial plumbing data.
There are many ways to extend the data by choosing bundle isomorphisms $\{\Phi_I^i\}_{i,I}$ extending the given ones. We want to extend in such a way that the full plumbing data is compatible with the splitting $\{g_i\}_{i\in\mathcal I}$. Because the splitting is itself symplectic, the compatible bundle isomorphisms can be chosen to be symplectic.

Now that the plumbing data is complete, we can choose compatible symplectic connections $\{\alpha_i\}_{i\in\mathcal I}$ (as constructed in Lemma \ref{nconnections}). The plumbing then inherits a symplectic form $\omega_{\mathcal I}$ (see Proposition \ref{symplecticformonplum}). We can invoke Proposition \ref{rigidembedding} to embed the plumbing in a rigid way, i.e. we obtain a map $F:N(\bigcup_iX_i)\rightarrow M$ such that $F|_{X_i}=id$ and $d_{X_i}F=g_i\times id$. This equality on the derivatives insures that $F^*(\omega)=\omega_{\mathcal I}$ along $\bigcup_i X_i$ (see Definition \ref{agreealong}). We can apply Lemma \ref{lemma1} to $F^*(\omega)$ and $\omega_{\mathcal I}$ on $N(\bigcup_i X_i)$ to get a map $\phi:N(\bigcup_i X_i)\rightarrow N(\bigcup_i X_i)$ such that $\phi|_{\bigcup_i X_i}=id$ and $\phi^*(F^*(\omega))=\omega_{\mathcal I}$.

Let $f=F\circ \phi:N(\bigcup_i X_i)\rightarrow M$.

\end{proof}

\begin{proof}[Proof of Theorem \ref{main}]

As in Proposition \ref{symplembedding}, we'll use the splitting maps $g_i:NX_i\rightarrow T_{X_i}M$ given by the orthogonal splitting on $T_{X_i}M\cong TX_i\oplus TX_i^{\perp_{\omega}}$.

We proceed by induction: for any $1\leq k\leq n$ one can build, for all $I$ such that $\# I=n-k$, a symplectic $k$-plumbing of $\{X_{I\cup\{i\}}\}_{i \notin I}$ in $X_{I}$, and then embed such plumbing symplectically into $X_{I}$.
For $k=1$, this is the usual symplectic neighborhood theorem, applied to $X_\mathcal I\subset X_{\mathcal I\backslash i}$.
For $k=n$, this yields the symplectic $n$-plumbing; Proposition \ref{symplembedding} then yields a symplectic embedding $F: N(\bigcup_i X_i)\rightarrow M$.

For the inductive step, assume that all $\{\psi_I^{I'}\}_{I'\subset I\subset \mathcal I, \#I'\geq n-k}$ have been constructed as symplectomorphisms with respect to $\alpha_{I'}=\oplus_{i\in I'}\alpha_i$. Similarly, assume  that all bundle isomorphisms $\{\Phi_I^{I'}\}_{I'\subset I\subset \mathcal I, \#I'\geq n-k}$ have been built compatibly.
We can then use Proposition \ref{symplembedding} to extend the system of bundle isomorphisms to $\{\Phi_I^{I'}\}_{I'\subset I\subset \mathcal I, \#I'\geq n-k-1}$ (compatible with $\{g_i\}_i$). Proposition \ref{symplembedding} also yields extensions of the connections, and corresponding symplectic embeddings $\{\psi_I^{I'}\}_{I'\subset I\subset \mathcal I, \#I'\geq n-k-1}$.

This concludes the induction.

\end{proof}

\begin{cor}[of Theorem \ref{main}]\label{maincor}
Let $\{X_i\}_{i\in\mathcal I}$ be a finite family of orthogonal symplectic submanifolds in $(M,\omega)$. Then there exist connections $\alpha_i$ on $NX_i$ and symplectic neighborhood embeddings $$\phi_i:(NX_i,\omega_{\alpha_i})\rightarrow (M,\omega)$$ which are pairwise compatible. This means that there are symplectic neighborhood embeddings $$\phi_{ij}:(NX_{ij},\omega_{\alpha_i+\alpha_j})\rightarrow (M,\omega)$$ which factor as $\phi_{ij}=\phi_i\circ \xi_{ij}^i=\phi_j\circ\xi_{ij}^j$, where $\xi_{ij}^i:NX_{ij}\cong NX_i\oplus NX_j \rightarrow NX_i$ and $\xi_{ij}^j:NX_{ij}\cong NX_i\oplus NX_j \rightarrow NX_j$ are bundle isomorphisms.
\end{cor}

\begin{proof}
This is just a rephrasing of Theorem \ref{main}: the functions $\phi_i$ are the restriction of $\varphi:(N_M(\bigcup_i X_i),\omega_{\mathcal I})\rightarrow (M,\omega)$ to each $NX_i\subset N_M(\bigcup_i X_i)$. The compatibility follows from the fact that such $\phi_i$ glue to a function on the plumbing. In particular one can construct $\phi_{ij}:=\phi_i\circ(\psi_{ij}^i)_*$ where $(\psi_{ij}^i)_*$ is defined in the previous discussion as a bundle isomorphism. So just let $\xi_{ij}^i:=(\psi_{ij}^i)_*, \xi_{ij}^j:=(\psi_{ij}^j)_*$.
\end{proof}

\section{G-equivariant formulation}

In this section we assume that a compact Lie group $G$ acts on $M$; we will show that the plumbing construction can be carried out equivariantly. Moreover, if $(M, \omega)$ is symplectic and $G$ acts by symplectomorphisms (not necessarily hamiltonian), we will show that the symplectic construction can also be carried out equivariantly.

\subsection{The case of one submanifold}

First of all let's review what happens in the standard case of one submanifold $X\subset M$.

\begin{thm}[G-equivariant tubular neighborhood]\label{equivnbhd}
Let $X\subset M$ be a submanifold preserved by the $G$-action. Consider the linearized action $G\circlearrowright NX$. There is a $G$-equivariant diffeomorphism $\phi:NX\rightarrow M$, defined close to the $0$-section.
\end{thm}

\begin{proof}
Follow the proof of Lemma \ref{lemma1}.
Usually, one would consider any metric $\lambda$ and construct the corresponding exponential map. For the exponential map to be $G$-equivariant we need a $G$-invariant metric. To construct one, take any metric $\lambda$ and consider $\tilde\lambda(v,w)=\int_G \lambda(g\cdot v,g\cdot w) dg$ where $dg$ is given by any measure on the group $G$, invariant by left translation (e.g. it can be the Haar measure). Let's check that $\tilde\lambda$ is $G$-invariant: $\tilde\lambda(h\cdot v),h\cdot w)=\int_G \lambda(h\cdot g\cdot v),h\cdot g\cdot w) d\nu=\int_G \lambda(g\cdot v,g\cdot w) dg=\tilde\lambda(v,w)$. Let's check that the correspondent exponential map is then $G$-equivariant: the key is that if $\tilde\lambda$ is $G$-invariant, then $G$ maps geodesics to geodesics. For $x\in X$, $v\in N_xX$, the point $exp(g\cdot x, g\cdot v)$ is given by flowing along a geodesic starting from $g\cdot x$ in the direction of $g\cdot v$ for time $|g\cdot v|=|v|$, which is the same as taking the geodesic at $x$ in the direction of $v$, then traslating it by the element $g$. Therefore, $exp(g\cdot x, g\cdot v)=g\cdot exp(x,v)$.
\end{proof}

Let's also consider the symplectic case:

\begin{thm}[G-equivariant symplectic tubular neighborhood]\label{tubularnbhd}
Let $X\subset (M,\omega)$ be a symplectic submanifold preserved by a symplectic $G$-action. Consider the linearized action $G\circlearrowright NX$. There is a $G$-equivariant symplectomorphism $\phi:NX\rightarrow M$, defined close to the $0$-section. The derivative of $\phi$ along $X$ is the identity.
\end{thm}

For which we need the following theorem.

\begin{thm}[G-equivariant symplectic neighborhood theorem]\label{esn}
Given a compact Lie group $G$ acting hamiltonially on the symplectic manifolds $Y,X_1,X_2$, and given equivariant symplectic embeddings $\iota_1:Y\rightarrow X_1$, $\iota_2:Y\rightarrow X_2$ plus a $G$-equivariant isomorphism $\phi:N_{X_1}Y\rightarrow N_{X_2}Y$ of symplectic normal bundles, there exists a $G$-equivariant symplectomorphism $\varphi:X_1\rightarrow X_2$ close to $Y$ such that
$$\varphi\circ\iota_1=\iota_2 \text{   and   } d_{\iota_1(Y)}\varphi=\phi.$$
\end{thm}

The proof of Theorem \ref{esn} is mostly based on Moser's argument:

\begin{lemma}[Equivariant Moser's argument]\label{lemma1_eq}
Let $\omega_0,\omega_1$ be symplectic forms on $X$, connected by a path of symplectic forms of fixed deRham class. Let there be an action $G\circlearrowright X$, hamiltonian with respect to both symplectic forms (and all along the path).
Assume that there exists a smooth family $\{\sigma_t\}_{t\in [0,1]}$ of 1-forms such that $\dd{t}\omega_t=\sigma_t$, and each $\sigma_t$ is $G$-equivariant. Then there exists a $G$-equivariant $\psi:X\rightarrow X$ such that $\psi^*(\omega_1)=\omega_0$.
\end{lemma}

\begin{proof}
Follow the standard proof of Moser's argument (for the non-equivariant case) and notice that since all $\omega_t$,$\sigma_t$ are $G$-equivariant then so is the vector field $X_t$ generating $\omega_t$, and similarly so is the flow of $X_t$ and so is $\psi$.
\end{proof}

\begin{proof}[Proof of Theorem \ref{esn}]
Recall that if $G\circlearrowright X_1$ and $Y$ is a $G$-invariant embedding, then there is an induced action 
$$G\circlearrowright NY \text{   s.t.   } g(y,v)=(g\cdot y, D_y g (v)).$$
Find a $G$-invariant metric on $X_1$, by averaging just like in the proof of Theorem \ref{equivnbhd}.
 The corresponding exponential map $exp_1:NY\rightarrow X_1$ is then $G$-equivariant.

Now assume that $\phi:N_{X_1}Y\rightarrow N_{X_2}Y$ is equivariant, i.e. it satisfies 
$$g(y,\phi_y(v))=(y,D_y g_2(\phi_y(v)))=(y,D_yg_1(v)) \forall g\in G, (y,v)\in N_{X_1}Y $$
where $g_1, g_2$ indicate the action of $g$ on $X_1,X_2$.

Consider $\tilde\varphi=exp_2\circ\phi\circ exp_1^{-1}:X_1\rightarrow X_2$ which is an equivariant diffeomorphism and a symplectomorphism because all three maps are. Then $\omega_1$ and $\tilde\varphi^* (\omega_2)$ are two symplectic forms agreeing along $Y$. We can thus construct a 1-form $\sigma$ such that $d\sigma=\omega_1-\tilde\varphi^* (\omega_2)$ as in \ref{lemma1}. Observe that such $\sigma$ is $G$-equivariant, therefore we can apply Lemma \ref{lemma1_eq} to find $\psi:X_1\rightarrow X_1$ such that $\psi^*\tilde\varphi^* (\omega_2)=\omega_1$ so $\varphi=\tilde\varphi\circ\psi$ is the map we need.
\end{proof}

Now we can prove Theorem \ref{tubularnbhd}.

\begin{proof}[Proof of Theorem \ref{tubularnbhd}]
First of all we need a $G$-invariant metric $\tilde\lambda$ which is compatible with $\omega$ at all points of $X$. As before (see Theorem \ref{equivnbhd}), we can construct $\tilde\lambda$ by first considering any metric $\lambda$ compatible with $\omega$, then averaging: $\tilde\lambda(v,w)=\int_G \lambda(g\cdot v,g\cdot w) dg$. Because $G$ acts symplectically, $\tilde\lambda$ is still compatible with $\omega$. Then the exponential map $exp: NX\rightarrow M$ is equivariant and $exp^*(\omega)=\omega$ along $X$. Consider any connection $\alpha$ on $NX$ and the induced symplectic form $\omega_\alpha$. Since $exp^*(\omega)=\omega_\alpha$ along $X$, we are in the setup of Theorem \ref{esn} and we can find a $G$-equivariant symplectomorphism $\phi: NX\rightarrow NX$ such that $\phi^*(exp^*(\omega))=\omega_\alpha$. Then $exp\circ\phi: NX\rightarrow M$ is as desired.
\end{proof}

\subsection{Equivariant plumbing}
In the presence of a $G$ action on $M$ preserving the $X_i$, we can construct the plumbing in such a way that it inherits a $G$ action.

\begin{prop}[Equivariant plumbing]\label{equivplum}
Let $G$ be a compact Lie group acting on $M$. Let $\{X_i\}_{i\in\mathcal I}$ be a finite collection of transverse, $G$-invariant submanifolds. Then there exist:
\begin{itemize}
\item a system of equivariant tubular neighborhoods $\psi_I^{I'}:N_{X_{I'}}X_I\rightarrow X_{I'}$ for $I\subset I'$;
\item a system of equivariant bundle isomorphisms $\Phi_I^{I'}:N_{X_{I'}}X_I\rightarrow X_{I'}$ for $I\subset I'$;
\item a natural $G$-action on the corresponding plumbing;
\item a $G$-equivariant embedding $\phi:N(\bigcup_i X_i)\rightarrow M$.
\end{itemize}
\end{prop}

Let's first prove a lemma.

\begin{lemma}\label{equibundles}
Assume the data of equivariant tubular neighborhoods $\psi_I^{I'}:N_{X_{I'}}X_I\rightarrow X_{I'}$ for $I'\subset I$ (as in Definition \ref{plumdata}). Then there exist equivariant bundle isomorphisms $\Phi_I^{I'}:N_{X_{I'}}X_I\rightarrow X_{I'}$ for $I\subset I'$ such that the corresponding plumbing naturally inherits a $G$-action.
\end{lemma}

\begin{proof}
Consider the $G$-action on $NX_{ij}$ obtained by linearizing the one on $M$: $g\cdot (x,v,w)=(gx,D_xg(v),D_xg(w))$. Then when we construct bundle isomorphisms $\Phi_I^{I'}:N_{X_{I'}}X_I\rightarrow X_{I'}$ for $I\subset I'$ we can ask that they are isomorphisms of $G$-bundles. 

Consider the plumbing of $\bigcup_i X_i$ given by such $G_i$-invariant plumbing data. We can check that the $G$-actions on $NX_i$ for all $i$ now descend to a $G$-action on the plumbing. The bundle map $\Phi_{ij}^i:NX_{ij}\rightarrow NX_i$ is fiberwise equivariant. In fact, it is an equivariant map of $G$-spaces, since: $\Phi_{ij}^i(z,v_j,v_i)=(\psi_{ij}^i(z,v_j),\tilde v_i)$ 
and we are assuming that $\psi_{ij}^i$ is equivariant.
The fact that $\Phi_{ij}^i$ is equivariant means that the $G$ action on $NX_i$ and $NX_j$ agree on the plumbing. This is true for all $i,j$, therefore the plumbing itself inherits a $G$ action.
\end{proof}

\begin{proof}[Proof of Proposition \ref{equivplum}]
All we need to do is go through the proof of Theorem \ref{smoothplumbingsexist} and check that the arguments can be used in the $G$-equivariant case. Given equivariant data, we can build the equivariant plumbing as explained in the previous lemma. 

Now let's go through the steps of Proposition \ref{smoothembedd} to see what happens. Recall that we need a choice of metric on the plumbing. In the equivariant setting, such metric has to be invariant. This isn't a problem since one can check that averaging $\lambda$ along $G$ doesn't change the required properties (in particular, if all $X_i$'s are geodesic and orthogonal with respect to $\lambda$, they still are with respect to $\tilde\lambda$ because $G$ preserves all the strata).

In the rest of the proof of Proposition \ref{smoothembedd}, we can ask that the bundle map $g:NX_{k+1}\hookrightarrow TM$ be $G$-equivariant. This is enough to ensure everything is equivariant, therefore the map $F:N(\bigcup_iX_i)\rightarrow M$ is also equivariant.
\end{proof}

An analog of Theorem \ref{plumbingsexist} is the following.

\begin{prop}\label{rigidequivplum}
Let $\{X_i\}_{i\in\mathcal{I}}$ be a finite set of transverse submanifolds of $M$. Let $G$ act on $M$ and assume each $X_i$ to be invariant under $G$.
Given an equivariant splitting $g_i:NX_i\rightarrow TM$ as in Definition \ref{compatiblesplitting}, assume that $g_i|_{N_{X_j}X_i\subset NX_i}:N_{X_j}X_i\rightarrow TX_j\subset TM$ for all $i,j$. One can find:

\begin{itemize}
\item plumbing data to build an equivariant plumbing $N(\bigcup_i X_i)$ compatible with the splitting;
\item an equivariant embedding $F: N(\bigcup_i X_i)\rightarrow M$ such that $d_{X_i}F=(g_i,id):T(NX_i)\rightarrow TM$.
\end{itemize} 
\end{prop}

\begin{proof}

We just need to check that each step of the proof of Theorem \ref{plumbingsexist} can be done equivariantly. Notice that, when constructing a plumbing compatible with the $g_i$'s, since each $g_i$ is equivariant we can easily impose that $\Phi_I^{I'}$ also is.

At this point, we need to show that the function $F$ produced by Proposition \ref{rigidembedding} is equivariant whenever all the input data (embedding and splitting and choice of smooth embedding) is. It's easy to see that both the map $\tilde\eta$ from Lemma \ref{straightening} and $\tilde {\tilde\eta}$ from Lemma \ref{extendstraightening} are equivariant in this setup.
Therefore $F$ is a composition of equivariant maps, hence it is equivariant.
\end{proof}

\subsection{Equivariant symplectic neighborhood of multiple submanifolds}
We finally get to the equivariant version of Theorem \ref{main}.

First we need the equivariant version of Lemma \ref{lemma1}.

\begin{lemma}\label{equivlemma1}
Let $\omega_1,\omega_2$ be symplectic forms on $M$ that agree along a collection of transversely intersecting symplectic submanifolds $X_i$. Let there be an action $G\circlearrowright M$ preserving each $X_i$. Then there exist open neighborhoods $\mathcal{U},\mathcal{V}$ of $\bigcup_i X_i$ and a $G$-equivariant symplectomorphism $\phi:\mathcal{U}\rightarrow\mathcal{V}$ such that
$$\phi|_{\bigcup_i X_i}=id \text{      and      } \phi^*(\omega_2)=\omega_1$$
\end{lemma}

\begin{proof}
Follow the proof of the non-equivariant case (Lemma \ref{lemma1}), but observe that at each step of the induction the construction of $\sigma_k$ is based on a choice of exponential map. As in the proof of Theorem \ref{esn}, we construct (by averaging) a $G$-equivariant metric $g$. Therefore the corresponding map $exp_k$ is $g$-equivariant, therefore so is the map $\phi_t$ for any $t$, therefore so is $\sigma_t$ for any $t$, therefore so is $\sigma_k$. This means all the $\omega^{(k)}$ are $G$-equivariant (meaning that $G$ acts symplectically with respect to each $\omega^{(k)}$), and all functions $f_k$ are also $G$-equivariant. So in the end, $\phi$ is equivariant.
\end{proof}

Finally we are able to construct an equivariant symplectic plumbing, and equivariantly embed it into $(M,\omega)$.

\begin{thm}[Equivariant symplectic plumbing]\label{equivsymplum}
Let $G$ be a compact Lie group acting on $(M,\omega)$ via symplectomorphisms. Let $\{X_i\}_{i\in\mathcal I}$ be a finite collection of orthogonal, $G$-invariant submanifolds. Then there exist:
\begin{itemize}
\item equivariant symplectic plumbing data $\{\psi_I^{I'}, \Phi_I^{I'}\}_{I'\subset I\subset \mathcal I}$;
\item a natural symplectic $G$-action on the plumbing;
\item a $G$-equivariant symplectic embedding $\phi:N(\bigcup_i X_i)\rightarrow M$.
\end{itemize}
\end{thm}

\begin{proof}
We need to merge Proposition \ref{equivplum} and Theorem \ref{main}. 

Consider the splitting $g_i:NX_i\rightarrow T_{X_i}M$ given by $\omega$. Since $G$ acts by symplectomorphisms, each $g_i$ must be equivariant. This means we can use Proposition \ref{rigidequivplum} instead of Proposition \ref{plumbingsexist} in the proof of Theorem \ref{main}.
Similarly, we can substitute Lemma \ref{lemma1} with Lemma \ref{equivlemma1}. When the data is equivariant, Lemma \ref{symplum} can produce a symplectic form $\omega_{\mathcal I}$ on the plumbing which is $G$-invariant: just choose each connection $\alpha_i$ to be $G$-equivariant on the $G$-bundle $NX_i$. After these modifications, the proof of Theorem \ref{main} will hold in the equivariant case!
\end{proof}

\begin{appendix}
    \section{A few words on positive intersection}

What happens if the $X_i$ that we consider are transverse symplectic submanifolds of $(M,\omega)$, but the intersection is not orthogonal? 

A nice formulation as in Corollary \ref{maincor} is not possible, as it implies the orthogonality of the collection.
Lemma \ref{lemma1} still holds, therefore the symplectic form on a neighborhood of $\bigcup_i X_i$ in $M$ is fully determined by the restriction of the symplectic form to $\bigcup_i X_i$.
On the other hand, there isn't a clear model for the form close to $\bigcup_i X_i$.

We could try to create a plumbing $N(\bigcup_i X_i)$ and build an ad hoc symplectic form $\omega_\mathcal I$ such that $\omega_\mathcal I(v,w)=\omega(v,w)$ for $v,w\in \bigcup_i T(X_i)$. A sketch of how to do it would be as follows:
\begin{itemize}
\item Construct the plumbing; produce a symplectic form on the plumbing as in Proposition \ref{symplum}, so that the intersection of the submanifolds is orthogonal;
\item Close to the intersection of $X_i, X_j$, define a new form such that the angle between $X_i, X_j$ is as wanted;
\item Interpolate the forms to obtain a symplectic form on the whole plumbing.
\end{itemize}

As it turns out, the interpolation is not a very easy task (as usual, when interpolating two symplectic forms, we need to worry about nondegeneracy). In particular, the suggested strategy fails unless one makes some assumptions.

The following definition of \textbf{positivity} is borrowed from \cite{mclean}, and it is a generalization of the definition for the intersection at a point:

\begin{defn}
Let $I=I_1\bigsqcup I_2 \subseteq \mathcal I$ (where $\bigsqcup$ indicates a disjoint union); then $TM|_{X_I}\cong TX_I\oplus N_1\oplus N_2$, where $N_i$ represents the symplectic normal bundle of $X_I$ in $X_{I_i}$.
Each bundle $TX_I$, $N_1$, $N_2$, $TM$ has an orientation induced by $\omega$. 

The intersection of $\{X_i\}_{i\in\mathcal I}$ is \textbf{positive} if, for all $I, I_1, I_2$ as above, the orientation induced by $\omega$ on $TM$ agrees with the induced orientation on $ TX_I\oplus N_1\oplus N_2$.
\end{defn}

Under this condition, one could try to define a symplectic plumbing to get an analog of Theorem \ref{main} for positive intersection. This doesn't seem interesting to pursue, since the most desirable local model is the one that only exists for orthogonal submanifolds.

The following result is attainable through interpolation of symplectic forms on a plumbing. It has been proved in \cite{mtz_nc}. Alsternatively, Lemma 5.3 of \cite{mclean} also can be used to show the following.

\begin{prop}\label{orthogdef}
If a collection of symplectic submanifolds intersect positively, then there is a symplectic isotopy, supported close to the intersection, deforming the symplectic submanifolds into orthogonally intersecting submanifolds. In particular, the isotopy preserves positivity of intersections at all times.
\end{prop}

The original paper \cite{mtz} also proves the following Theorem, which we can now interpret as a consequence of Theorem \ref{main} and Proposition \ref{orthogdef}:

\begin{thm}
Given a family $\{X_i\}_{i\in\mathcal I}$ of symplectic submanifolds with positive intersection, $\omega$ can be deformed in a neighborhood of $\bigcup_i X_i$ in such a way that a neighborhood of $\bigcup_i X_i$ becomes symplectomorphic to $N_M(\bigcup_i X_i)$. 
\end{thm}

\begin{rmk}
The result discussed in this Appendix is not a result connected to the flexibility of symplectic forms: a metric could also be deformed close to an intersection to assume a standard model. On the other hand, Theorem \ref{main} is a result of the flexibility of symplectic forms.
\end{rmk}

\end{appendix}

\bibliography{mybib}
\bibliographystyle{alpha}

\end{document}